\documentclass[12pt,psamsfonts]{amsart}
\usepackage[left=3cm,top=2.5cm,right=3cm,bottom=2.5cm]{geometry}

\usepackage{amsmath,amssymb}
\usepackage{graphicx}
\usepackage{color}
\usepackage{psfrag} 
\usepackage{overpic} 

\newtheorem{theorem}{Theorem}
\newtheorem{proposition}[theorem]{Proposition}
\newtheorem{corollary}[theorem]{Corollary}

\newtheorem {remark}[theorem]{Remark}

\title[]{Bifurcation of Limit Cycles from a Periodic Annulus Formed by a Center and Two Saddles in Piecewise Linear Differential System with Three Zones}
\author[C. Pessoa and R. Ribeiro]{}


\begin{document}
 \maketitle

\centerline{\scshape  Claudio Pessoa and Ronisio Ribeiro}
\medskip

{\footnotesize \centerline{Universidade Estadual Paulista (UNESP),} \centerline{Instituto de Bioci\^encias Letras e Ci\^encias Exatas,} \centerline{R. Cristov\~ao Colombo, 2265, 15.054-000, S. J. Rio Preto, SP, Brazil }
\centerline{\email{c.pessoa@unesp.br} and \email{ronisio.ribeiro@unesp.br}}}

\medskip

\bigskip

\begin{quote}{\normalfont\fontsize{8}{10}\selectfont
{\bfseries Abstract.} In this paper, we study the number of limit cycles that can bifurcate from a periodic annulus in discontinuous planar piecewise linear Hamiltonian differential system with three zones separated by two parallel straight lines, such that the linear differential systems that define the piecewise one have a center and two saddles. That is, the linear differential system in the region between the two parallel lines (i.e. the central subsystem) has a center and the others subsystems have saddles. We prove that if the central subsystem has a real or a boundary center, then we have at least six limit cycles bifurcating from the periodic annulus by linear perturbations, four passing through the three zones and two passing through the two zones. Now, if the central subsystem has a virtual center, then we have at least four limit cycles 
bifurcating from the periodic annulus by linear perturbations, three passing through the three zones and one passing through the two zones. For this, we obtain a normal form for these piecewise Hamiltonian systems and study the number of zeros of its Melnikov functions defined in two and three zones.

\par}
\end{quote}

\section{Introduction and Main Results}
The problem of determine the number and position of limit cycles of polynomial differential systems was proposed by Hilber in 1900 as part of Hilbert’s 16th problem, see \cite{Hil02}. Currently this problem has been considered for piecewise differential systems. This class of differential systems have piqued the attention of researchers in qualitative theory of differential equations, mainly because many phenomena can be described by these class of models, for instance in mechanics, electrical circuits, control theory, neurobiology, etc (see the book \cite{diB08} and the papers \cite{Chu90, Fit61, McK70, Nag62}).

For continuous planar piecewise differential systems with two zones, was proved that such systems have at most one limit cycle (see \cite{Fre98}). In recent years, numerous studies about discontinuous piecewise linear differential systems have arisen in the literature.  The case with two zones separated by one straight line is the simplest and most explored, however, the maximum number of limit cycles is not known, but important partial results have been obtained, see for instance \cite{Bra13, Buz13, Fre12, Fre14b, Li14, Lli12, LNT, Wan19}. There are few works when the piecewise system is defined in more than two zones, see for instance \cite{Don17, Hu13, Lli15b, Wan16, Yan20}. Most results involving lower bounds for the number of the limit cycles are obtained by imposing restrictive hypotheses on the systems, such as symmetry and linearity. For instance, in \cite{Lli14}, conditions for nonexistence and existence of one, two or three limit cycles have been obtained, for symmetric continuous piecewise linear differential systems with three zones. Examples with two limit cycles surrounding a unique singular point at the origin was found in \cite{Lli15, Lim17}, for the nonsymmetric case. The best lower bound for the number of limit cycles from a discontinuous piecewise linear differential system with three zones separated by two parallel straight lines, without these restrictions, is seven. This lower bound was obtained by linear perturbations of a piecewise linear differential system with subsystems without singular points and a boundary pseudo--focus, see \cite{Xio21}. As far as we know, all other papers that estimate the number of limit cycles for these class of piecewise linear differential systems have found at most 1 or 3 limit cycles, see \cite{Fon20, Li21, Lli18a, Pes22a, Xio20}. The search for upper bound for the number of limit cycles that a piecewise linear system with three zones can have is what motivates most of the recently works found in the literature about this topic. However, all cases are interesting in themselves, that is, the search for better quotas for number of limit cycles can not be used to discourage the study of particular families. So the type of singular points of the subsystems and their positions, that is, whether they are real, virtual or boundary, are important questions in this study and should be considered for all subclass of piecewise linear systems with three zones. 

In this paper, we contribute along these lines. Our goal is estimated the lower bounds for the number of crossing limit cycles of a discontinuous piecewise linear near-Hamiltonian differential systems with three zones, given by

\begin{equation}\label{eq:01}
	\left\{\begin{array}{ll}
		\dot{x}= H_y(x,y)+\epsilon f(x,y), \\
		\dot{y}= -H_x(x,y)+\epsilon g(x,y),
	\end{array}
	\right.
\end{equation}
with
\begin{equation*}
	H(x,y)=\left\{\begin{array}{ll}\vspace{0.2cm}
		H^{\scriptscriptstyle L}(x,y)=\dfrac{b_{\scriptscriptstyle L}}{2}y^2-\dfrac{c_{\scriptscriptstyle L}}{2}x^2+a_{\scriptscriptstyle L}xy+\alpha_{\scriptscriptstyle L}y-\beta_{\scriptscriptstyle L}x, \quad x\leq -1, \\ \vspace{0.2cm}
		H^{\scriptscriptstyle C}(x,y)=\dfrac{b_{\scriptscriptstyle C}}{2}y^2-\dfrac{c_{\scriptscriptstyle C}}{2}x^2+a_{\scriptscriptstyle C}xy+\alpha_{\scriptscriptstyle C}y-\beta_{\scriptscriptstyle C}x, \quad -1\leq x\leq 1, \\
		H^{\scriptscriptstyle R}(x,y)=\dfrac{b_{\scriptscriptstyle R}}{2}y^2-\dfrac{c_{\scriptscriptstyle R}}{2}x^2+a_{\scriptscriptstyle R}xy+\alpha_{\scriptscriptstyle R}y-\beta_{\scriptscriptstyle R}x, \quad x \geq 1, \\
	\end{array}
	\right.
\end{equation*}
\begin{equation}\label{eq:02}
	f(x,y)=\left\{\begin{array}{ll}
		f_{\scriptscriptstyle L}(x,y)=r_{10}x+r_{01}y+r_{00}, \quad x\leq -1, \\
		f_{\scriptscriptstyle C}(x,y)=u_{10}x+u_{01}y+u_{00}, \quad -1\leq x\leq 1, \\
		f_{\scriptscriptstyle R}(x,y)=p_{10}x+p_{01}y+p_{00}, \quad x \geq 1, \\
	\end{array}
	\right.
\end{equation}
\begin{equation}\label{eq:03}
	g(x,y)=\left\{\begin{array}{ll}
		g_{\scriptscriptstyle L}(x,y)=s_{10}x+s_{01}y+s_{00}, \quad x\leq -1, \\
		g_{\scriptscriptstyle C}(x,y)=v_{10}x+v_{01}y+v_{00}, \quad -1\leq x\leq 1, \\
		g_{\scriptscriptstyle R}(x,y)=q_{10}x+q_{01}y+q_{00}, \quad x \geq 1, \\
	\end{array}
	\right.
\end{equation}
where the dot denotes the derivative with respect to the independent variable $t$, here called the time, and $0\leq\epsilon<<1$.  When $\epsilon=0$ we say that system \eqref{eq:01} is a piecewise Hamiltonian differential system. We call system \eqref{eq:01} of {\it left subsystem} when $x\leq -1$, {\it right subsystem} when $x\geq 1$ and {\it central subsystem} when $-1\leq x\leq 1$. Moreover, we can classify a singular point $p$ of a subsystem from \eqref{eq:01} according to its position in relation to the subsystem. More precisely, the central subsystem from \eqref{eq:01} has a {\it real singular point} $p=(p_x,p_y)$ when $-1<p_x<1$, a {\it virtual singular point} when $p_x>1$ or $p_x<-1$ and a {\it boundary singular point} when $p_x=\pm1$. Accordingly, the left (resp. right) subsystem from \eqref{eq:01} has a real singular point $p=(p_x,p_y)$ when $p_x<-1$ (resp. when $p_x>1$ ), a virtual singular point when $p_x>-1$ (resp. when $p_x<1$) and a boundary singular point when $p_x=-1$ (resp. when $p_x=1$).

\medskip

In addition, let us assume that system $\eqref{eq:01}|_{\epsilon=0}$ satisfies the following hypotheses:
\begin{itemize}
	\item[{\rm (H1)}] The unperturbed central subsystem from $\eqref{eq:01}|_{\epsilon=0}$ has a center and the others unperturbed subsystems from $\eqref{eq:01}|_{\epsilon=0}$ have saddles.
	\item[{\rm (H2)}] The unperturbed system  $\eqref{eq:01}|_{\epsilon=0}$ has only crossing points on the straights lines $x=\pm 1$, except by some tangent points.
	\item[{\rm (H3)}] The unperturbed system $\eqref{eq:01}|_{\epsilon=0}$ has a periodic annulus consisting of a family of crossing periodic orbits such that each orbit of this family passes thought the three or two zones with clockwise orientation.  	
\end{itemize}

Thus, the main result of this paper is the follow.

\begin{theorem}\label{the:01}
	The number of limit cycles of system \eqref{eq:01}, satisfying hypotheses {\rm (Hi)} for $i=1,2,3$, which can bifurcate from the periodic annulus of the unperturbed system $\eqref{eq:01}|_{\epsilon=0}$ is either at least six if the central subsystem has a real/boundary center or at least four if the central subsystem has a virtual center.
\end{theorem}

The paper is organized as follows. In Section \ref{sec:NF}, we obtain a normal form to system $\eqref{eq:01}|_{\epsilon=0}$ that simplifies the compute. In order to prove the Theorem \ref{the:01}, we will study the number of zeros of the first order Melnikov function associated to system \eqref{eq:01} in Section \ref{sec:mel} (see also \cite{ Xio20, Xio21} for more details about the Melnikov function). Our study is concentrated in the neighborhood of the periodic orbit that simultaneously bounded the periodic annulus from the two and three zones, since to estimate the zeros of the Melnikov function we consider its expansion at the point corresponding to this orbit. What distinguishes our approach apart from others is that we provide a detailed analytical method to study the number of simultaneous zeros from Melnikov functions defined in two and three zones (see Proposition \ref{the:coef}). Finally, in Section \ref{sec:Teo} we will prove Theorem \ref{the:01}.


\section{Normal Form}\label{sec:NF}
In order to decrease the number of parameters of system $\eqref{eq:01}|_{\epsilon=0}$, we will do a continuous linear change of variables that keeps invariant the straight lines $x=\pm 1$. This change of variables is a homeomorphism and will be a topological equivalence between the systems. More precisely, we have the follow result.

\begin{proposition}\label{fn:01}
	Suppose that the central subsystem from $\eqref{eq:01}|_{\epsilon=0}$ has a center and the other two subsystems have two saddles. Then, after a linear change of variables and a rescaling of the independent variable, we can assume that $\alpha_{\scriptscriptstyle L}=a_{\scriptscriptstyle L}$, $\alpha_{\scriptscriptstyle R}=-a_{\scriptscriptstyle R}$, $b_{\scriptscriptstyle C}=1$, $c_{\scriptscriptstyle C}=-1$ and $a_{\scriptscriptstyle C}=\alpha_{\scriptscriptstyle C}=0$.
\end{proposition}

\begin{proof}
	As the central subsystem from $\eqref{eq:01}|_{\epsilon=0}$ has a center with clockwise orientation of the orbits, then $a_{\scriptscriptstyle C}^2+b_{\scriptscriptstyle C}c_{\scriptscriptstyle C}<0$ and $b_{\scriptscriptstyle C}>0$. Moreover, $b_i\ne 0$, for $i=L,R$. In fact, if $b_i=0$ then the saddle of right or left subsystem have a separatrix parallel to straight line $x=0$. Note that, system $\eqref{eq:01}|_{\epsilon=0}$ has four tangent points given by $P_1=(1,-(a_{\scriptscriptstyle C}+\alpha_{\scriptscriptstyle C})/b_{\scriptscriptstyle C})$, $P_2=(1,-(a_{\scriptscriptstyle R}+\alpha_{\scriptscriptstyle R})/b_{\scriptscriptstyle R})$, $P_3=(-1,(a_{\scriptscriptstyle C}-\alpha_{\scriptscriptstyle C})/b_{\scriptscriptstyle C})$ and $P_4=(-1,(a_{\scriptscriptstyle L}-\alpha_{\scriptscriptstyle L})/b_{\scriptscriptstyle L})$. By hypothesis (H2), we have that the system $\eqref{eq:01}|_{\epsilon=0}$ have only crossing points on the straight lines $x=\pm 1$, except in the tangent points. Hence, for all $y\in\mathbb{R}\setminus\{(\pm a_{\scriptscriptstyle C}-\alpha_{\scriptscriptstyle C})/b_{\scriptscriptstyle C},-(a_{\scriptscriptstyle R}+\alpha_{\scriptscriptstyle R})/b_{\scriptscriptstyle R}),(a_{\scriptscriptstyle L}-\alpha_{\scriptscriptstyle L})/b_{\scriptscriptstyle L})\}$, we must have
	$$\left\langle  X_{\scriptscriptstyle L}(-1,y),(1,0)\right\rangle \left\langle  X_{\scriptscriptstyle C}(-1,y),(1,0)\right\rangle>0\quad\text{and}\quad\left\langle  X_{\scriptscriptstyle R}(1,y),(1,0)\right\rangle \left\langle  X_{\scriptscriptstyle C}(1,y),(1,0)\right\rangle>0.$$ 
	But this implies that $b_{\scriptscriptstyle L}b_{\scriptscriptstyle C}>0$, $b_{\scriptscriptstyle R}b_{\scriptscriptstyle C}>0$, $P_1=P_2$ and $P_3=P_4$. Therefore, as $b_{\scriptscriptstyle C}>0$, we have that
	\begin{equation}\label{ch:01}
		\alpha_{\scriptscriptstyle L}=\frac{a_{\scriptscriptstyle L}b_{\scriptscriptstyle C}+b_{\scriptscriptstyle L}(\alpha_{\scriptscriptstyle C}-a_{\scriptscriptstyle C})}{b_{\scriptscriptstyle C}},\quad b_{\scriptscriptstyle L}>0,\quad \alpha_{\scriptscriptstyle R}=\frac{-a_{\scriptscriptstyle R}b_{\scriptscriptstyle C}+b_{\scriptscriptstyle R}(a_{\scriptscriptstyle C}+\alpha_{\scriptscriptstyle C})}{b_{\scriptscriptstyle C}}\quad\text{and}\quad b_{\scriptscriptstyle R}>0.
	\end{equation}
	Assuming the conditions \eqref{ch:01}, consider the change of variables  
	\begin{displaymath}
		\left(\begin{array}{c}
			u\\
			v
		\end{array}\right)=\left(\begin{array}{cc}
			1 & 0\vspace{0.2cm}\\
			\dfrac{a_{\scriptscriptstyle C}}{\omega_{\scriptscriptstyle C}} &\,\,\,\,\, \dfrac{b_{\scriptscriptstyle C}}{\omega_{\scriptscriptstyle C}}
		\end{array}\right)\left(\begin{array}{c}
			x\\
			y
		\end{array}\right)+\left(\begin{array}{c}
			0\vspace{0.2cm}\\
			\dfrac{\alpha_{\scriptscriptstyle C}}{\omega_{\scriptscriptstyle C}}
		\end{array}\right),
	\end{displaymath}
	with $\omega_{\scriptscriptstyle C}=\sqrt{-a^2_{\scriptscriptstyle C} - b_{\scriptscriptstyle C}c_{\scriptscriptstyle C}}$. Applying this change of variables and rescaling the time by $\tilde{t}=\omega_{\scriptscriptstyle C}\,t$, we obtain the results after rewriting the parameters.	
\end{proof}

\begin{remark}\label{rem:beta}
	Consider the system $\eqref{eq:01}|_{\epsilon=0}$ in its normal form, i.e. with  $\alpha_{\scriptscriptstyle L}=a_{\scriptscriptstyle L}$, $\alpha_{\scriptscriptstyle R}=-a_{\scriptscriptstyle R}$, $b_{\scriptscriptstyle C}=c_{\scriptscriptstyle C}=1$ and $a_{\scriptscriptstyle C}=\alpha_{\scriptscriptstyle C}=0$. Note that when $\beta_{\scriptscriptstyle C}=0$, we have that the singular point of the central subsystem from $\eqref{eq:01}|_{\epsilon=0}$ is at the origin. In this case, assuming the hypotheses (Hi), i=1,2,3, the period annulus of system $\eqref{eq:01}|_{\epsilon=0}$ has all its periodic orbits passing through the three zones bounded by the orbit $L_0$ of the central subsystem tangent to straight lines $x=\pm 1$ in the points $P_{\scriptscriptstyle R}=(1,0)$ and $P_{\scriptscriptstyle L}=(-1,0)$ (see Fig. \ref{fig:07} (a)).
	
	If $\beta_{\scriptscriptstyle C}\ne0$, after a reflection around the straight line $x=0$ (if necessary), we can assuming without loss of generality that $\beta_{\scriptscriptstyle C}>0$. In this case, the period annulus of system $\eqref{eq:01}|_{\epsilon=0}$ has all its periodic orbits passing through the three zones bounded by the orbit $L_0^1$ that is obtained by the orbit of the central subsystem tangent to straight lines $x= -1$ in the point $P_{\scriptscriptstyle L}=(-1,0)$. Note that, $ L_0^1$ intercept crosswise the straight line $x= 1$ in two distinct points. Moreover, the periodic annulus has periodic orbits passing by two zones, which are bounded by only $L_0^1$ when $\beta_{\scriptscriptstyle C}\geq1$, or bounded by $L_0^1$ and $ L_{0}^2$ when  $0<\beta_{\scriptscriptstyle C}<1$, where $ L_{0}^2$ is the orbit obtained by the orbit of the central subsystem tangent to straight lines $x= 1$ in the point $P_{\scriptscriptstyle R}=(1,0)$ contained in the region bounded by $L_0^1$ (see Fig. \ref{fig:07} (b)--(c)).
	
	The case $\beta_{\scriptscriptstyle C}=0$ was studied in the paper \cite{Pes22a}, where the authors showed that at least 3 limit cycles can bifurcate from a periodic annulus  perturbed by piecewise linear polynomials. In this paper, we will study only the case $\beta_{\scriptscriptstyle C}>0$, and we will show that either at least six limit cycles can bifurcate from periodic annulus near the periodic orbit $ L_{0}^1$ when $0<\beta_{\scriptscriptstyle C}\leq1$ or at least four limit cycles can bifurcate from periodic annulus near the periodic orbit $ L_{0}^1$ when $\beta_{\scriptscriptstyle C}>1$. 
	
	\begin{figure}[h!]
		\begin{center}		
			\begin{overpic}[width=5in]{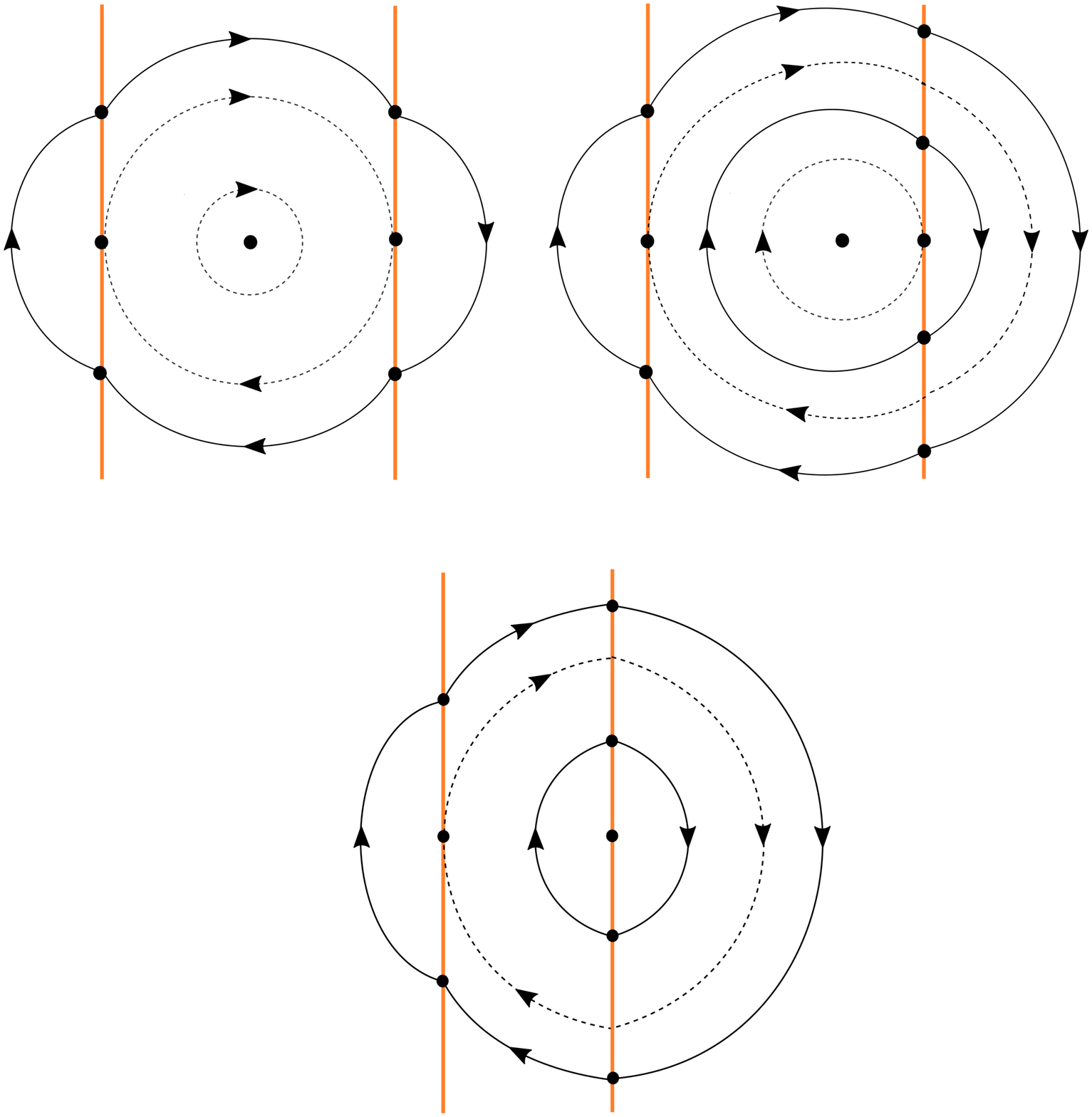}
				\put(6,53) {$x=1$}
				\put(32,53) {$x=-1$}
				\put(78,53) {$x=-1$}
				\put(55,53) {$x=1$}	
				\put(36,-3) {$x=-1$}
				\put(53,-3) {$x=1$}		
				\put(37,77) {$P_{\scriptscriptstyle R}$}
				\put(5,77) {$P_{\scriptscriptstyle L}$}
				\put(27,90) {$L_0$}
				\put(53,77) {$P_{\scriptscriptstyle L}$}
				\put(83,77) {$P_{\scriptscriptstyle R}$}
				\put(78,93.5) {$L_0^1$}
				\put(77,85) {$L_{0}^2$}
				\put(35,25) {$P_{\scriptscriptstyle L}$}
				\put(56,25) {$P_{\scriptscriptstyle R}$}
				\put(50,42) {$L_0^1$}
				\put(21,51) {(a)}
				\put(70,51) {(b)}
				\put(48,-4) {(c)}
			\end{overpic}
		\end{center}
		\vspace{0.7cm}
		\caption{Periodic orbits tangent to straight lines $x=\pm 1$ of system $\eqref{eq:01}|_{\epsilon=0}$ with  $\alpha_{\scriptscriptstyle L}=a_{\scriptscriptstyle L}$, $\alpha_{\scriptscriptstyle R}=-a_{\scriptscriptstyle R}$, $b_{\scriptscriptstyle C}=c_{\scriptscriptstyle C}=1$ and $a_{\scriptscriptstyle C}=\alpha_{\scriptscriptstyle C}=0$ when $\beta_{\scriptscriptstyle C}=0$ (a),  $0<\beta_{\scriptscriptstyle C}<1$ (b) and $\beta_{\scriptscriptstyle C}\geq1$ (c).}\label{fig:07}
	\end{figure} 
\end{remark}


\section{Melnikov Function}\label{sec:mel}
In this section, we will introduce the first order Melnikov function associated to system $\eqref{eq:01}$ which will be needed  to prove the main result of this paper.

For this purpose, suppose that unperturbed system $\eqref{eq:01}|_{\epsilon=0}$ satisfies hypothesis (H3), i.e. there are a open interval $J_0=(\alpha_0,\beta_0)$ and a periodic annulus consisting of a family of crossing periodic orbits $L^0_h$, with $h\in J_0$, such that each orbit of this family crosses the straight lines $x=\pm 1$ in four points, $A(h)=(1,a(h))$, $A_1(h)=(1,a_1(h))$, $A_2(h)=(-1,a_2(h))$ and $A_3(h)=(-1,a_3(h))$, with $a_1(h)<a(h)$ and $a_2(h)<a_3(h)$, thought the three zones with clockwise orientation (see Fig. \ref{fig:01}), satisfying the following equations
\begin{equation*}\label{eq:050}
	\begin{aligned}
		& H^{\scriptscriptstyle R}(A(h))=H^{\scriptscriptstyle R}(A_1(h)), \\
		& H^{\scriptscriptstyle C}(A_1(h))=H^{\scriptscriptstyle C}(A_2(h)), \\
		& H^{\scriptscriptstyle L}(A_2(h))=H^{\scriptscriptstyle L}(A_3(h)), \\
		& H^{\scriptscriptstyle C}(A_3(h))=H^{\scriptscriptstyle C}(A(h)), 
	\end{aligned}
\end{equation*} 
and, for $h\in J_0$,
$$H^{\scriptscriptstyle R}_y(A(h))\,H^{\scriptscriptstyle R}_y(A_1(h))\,H^{\scriptscriptstyle L}_y(A_2(h))\,H^{\scriptscriptstyle L}_y(A_3(h))\ne 0,$$ 
$$H^{\scriptscriptstyle C}_y(A(h))\,H^{\scriptscriptstyle C}_y(A_1(h))\,H^{\scriptscriptstyle C}_y(A_2(h))\,H^{\scriptscriptstyle C}_y(A_3(h))\ne 0.$$
Moreover, for each $h\in J_0$, system $\eqref{eq:01}|_{\epsilon=0}$ has a crossing periodic orbit $L^0_h=\widehat{AA_1}\cup\widehat{A_1A_2}\cup\widehat{A_2A_3}\cup\widehat{A_3A}$, passing through these points (see Fig. \ref{fig:01}).	

We also have, by hypothesis (H3), that there is a open interval $J_1=(\alpha_1,\beta_1)$ such that the unperturbed system $\eqref{eq:01}|_{\epsilon=0}$ has a periodic annulus consisting of a family of crossing periodic orbits $L_h^1$,  $h\in J_1$, and each orbit of this family crosses the straight lines $x=1$ in two points, $B(h)=(1,b(h))$ and $B_1(h)=(1,b_1(h))$, with $b_1(h)<b(h)$, through the two zones with clockwise orientation (see Fig. \ref{fig:01}), satisfying the following equations
\begin{equation*}\label{eq:051}
	\begin{aligned}
		& H^{\scriptscriptstyle R}(B(h))=H^{\scriptscriptstyle R}(B_1(h)), \\
		& H^{\scriptscriptstyle C}(B_1(h))=H^{\scriptscriptstyle C}(B(h)),
	\end{aligned}
\end{equation*}
and, for $h\in J_1$,
$$H^{\scriptscriptstyle R}_y(B(h))\,H^{\scriptscriptstyle R}_y(B_1(h))\,H^{\scriptscriptstyle C}_y(B(h))\,H^{\scriptscriptstyle C}_y(B_1(h))\ne 0.$$ 
Moreover, for each $h\in J_1$, system $\eqref{eq:01}|_{\epsilon=0}$ has a crossing periodic orbit $L^1_h=\widehat{BB_1}\cup\widehat{B_1B}$, passing through these points (see Fig. \ref{fig:01}).

\begin{figure}[h]
	\begin{center}		
		\begin{overpic}[width=3in]{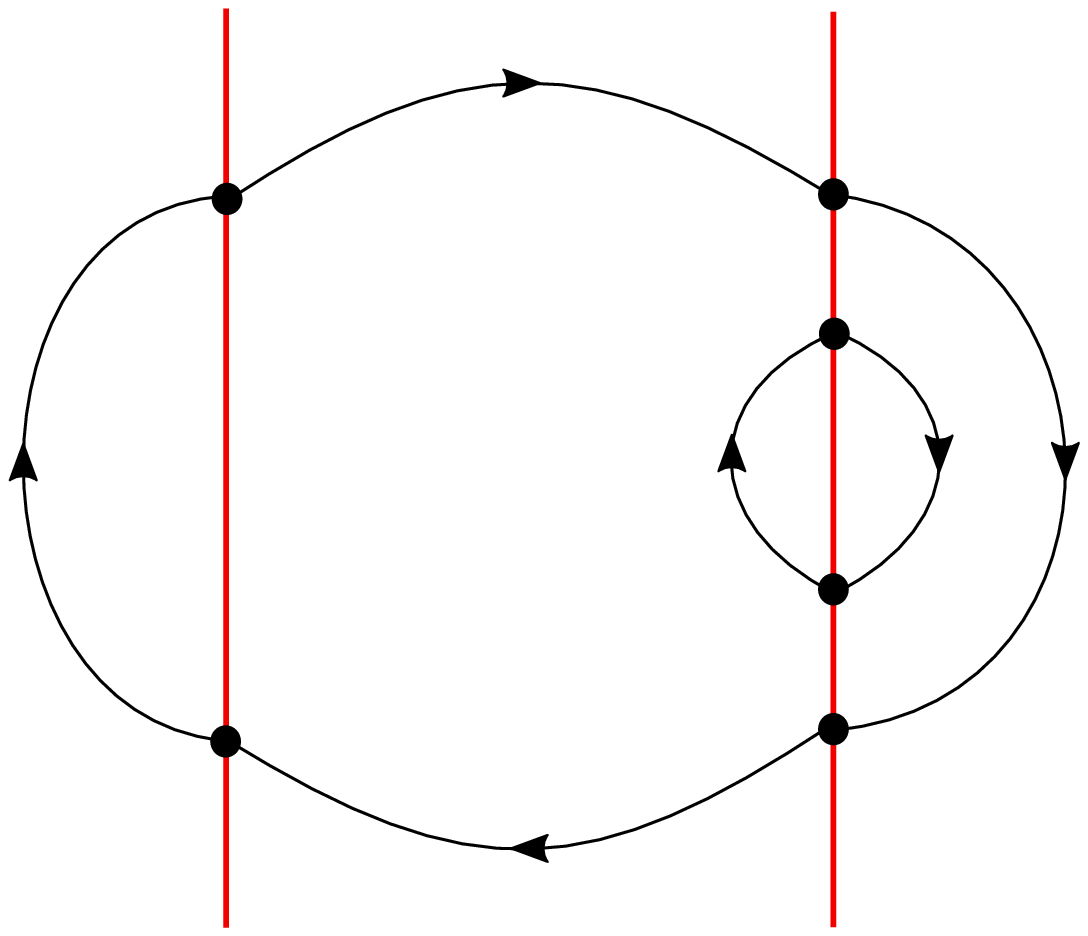}
			\put(71,-4) {$x=1$}
			\put(12,-4) {$x=-1$}
			\put(79,70) {$A(h)$}
			\put(79,14) {$A_1(h)$}
			\put(7,12) {$A_2(h)$}
			\put(7,70) {$A_3(h)$}
			\put(65,57) {$B(h)$}
			\put(64,27) {$B_1(h)$}
			\put(58,79) {$L_h^0$}
			\put(86,50) {$L_h^1$}
		\end{overpic}
	\end{center}
	\vspace{0.7cm}
	\caption{The crossing periodic orbit of system  $\eqref{eq:01}|_{\epsilon=0}$.}\label{fig:01}
\end{figure} 

Consider, for $h\in J_0$, the solution of right subsystem from \eqref{eq:01} starting at point $A(h)$. Let $A_{1\epsilon}(h)=(1,a_{1\epsilon}(h))$ be the first intersection point of this orbit with straight line $x=1$. Denote by $A_{2\epsilon}(h)=(-1,a_{2\epsilon}(h))$ the first intersection point of the orbit of central subsystem from \eqref{eq:01} starting at $A_{1\epsilon}(h)$ with straight line $x=-1$, $A_{3\epsilon}(h)=(-1,a_{3\epsilon}(h))$ the first intersection point of the orbit of left subsystem from \eqref{eq:01} starting at $A_{2\epsilon}(h)$ with straight line $x=-1$ and $A_{\epsilon}(h)=(1,a_{\epsilon}(h))$ the first intersection point of the orbit of central subsystem from \eqref{eq:01} starting at $A_{3\epsilon}(h)$ with straight line $x=1$ (see Fig. \ref{fig:02}). Now, for $h\in J_1$, consider  the solution of right subsystem from \eqref{eq:01} starting at point $B(h)$. Let $B_{1\epsilon}(h)=(1,b_{1\epsilon}(h))$ be the first intersection point of this orbit with straight line $x=1$ and  $B_{\epsilon}(h)=(1,b_{\epsilon}(h))$ the first intersection point of the orbit of central subsystem from \eqref{eq:01} starting at $B_{1\epsilon}(h)$ with straight line $x=1$ (see Fig. \ref{fig:02}).

\begin{figure}[h]
	\begin{center}		
		\begin{overpic}[width=3in]{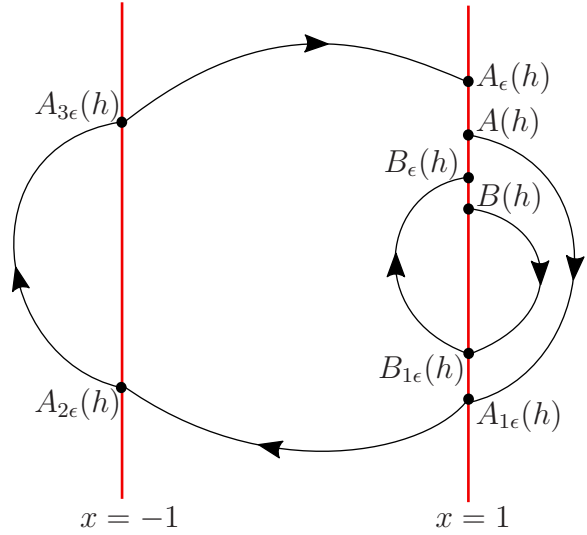}
			\put(74,-4) {$x=1$}
			\put(12,-4) {$x=-1$}
			\put(81,65) {$A(h)$}
			\put(81,14) {$A_{1\epsilon}(h)$}
			\put(4,16) {$A_{2\epsilon}(h)$}
			\put(4,68) {$A_{3\epsilon}(h)$}
			\put(81,73) {$A_{\epsilon}(h)$}
			\put(81,52) {$B(h)$}
			\put(64,22) {$B_{1\epsilon}(h)$}
			\put(65,58) {$B_{\epsilon}(h)$}
		\end{overpic}
	\end{center}
	\vspace{0.7cm}
	\caption{Poincaré maps of system  \eqref{eq:01}.}\label{fig:02}
\end{figure} 
We define the Poincaré maps of piecewise system \eqref{eq:01} as follows, 
\begin{equation*}
	\begin{aligned}
		& H^{\scriptscriptstyle R}(A_{\epsilon}(h))-H^{\scriptscriptstyle R}(A(h))=\epsilon M_0(h)+\mathcal{O}(\epsilon^2),\quad \forall h\in J_0,\\
		& H^{\scriptscriptstyle R}(B_{\epsilon}(h))-H^{\scriptscriptstyle R}(B(h))=\epsilon M_1(h)+\mathcal{O}(\epsilon^2),\quad \forall h\in J_1,
	\end{aligned}
\end{equation*}
where $M_0(h)$ and $M_1(h)$ are called the {\it first order Melnikov functions} associated to piecewise system \eqref{eq:01} in three and two zones, respectively. Then, following the steps of the proof of the Theorem 1.1 in \cite{Liu10} for the case in two zones and, doing the obvious adaptations, for the case of three zones, we can prove the following theorem (see also \cite{Xio21}).
\begin{theorem}\label{teo:mel}
	Consider  system \eqref{eq:01} with $0\leq \epsilon <<1$ and suppose that the unperturbed system $\eqref{eq:01}|_{\epsilon=0}$ satisfies the hypotheses $(\text{H}i)$, $i=1,2,3$. Then the first order Melnikov functions associated to system \eqref{eq:01} can be expressed as
	\begin{equation}\label{eq:mel0}
		\begin{aligned}
			M_0(h) &  = \frac{H_y^{\scriptscriptstyle R}(A)}{H_y^{\scriptscriptstyle C}(A)} I_{\scriptscriptstyle C}^0 + \frac{H_y^{\scriptscriptstyle R}(A)H_y^{\scriptscriptstyle C}(A_3)}{H_y^{\scriptscriptstyle C}(A)H_y^{\scriptscriptstyle L}(A_3)} I_{\scriptscriptstyle L}^0 + \frac{H_y^{\scriptscriptstyle R}(A)H_y^{\scriptscriptstyle C}(A_3)H_y^{\scriptscriptstyle L}(A_2)}{H_y^{\scriptscriptstyle C}(A)H_y^{\scriptscriptstyle L}(A_3)H_y^{\scriptscriptstyle C}(A_2)} \bar{I}_{\scriptscriptstyle C}^0 \\
			&\quad  +  \frac{H_y^{\scriptscriptstyle R}(A)H_y^{\scriptscriptstyle C}(A_3)H_y^{\scriptscriptstyle L}(A_2)H_y^{\scriptscriptstyle C}(A_1)}{H_y^{\scriptscriptstyle C}(A)H_y^{\scriptscriptstyle L}(A_3)H_y^{\scriptscriptstyle C}(A_2)H_y^{\scriptscriptstyle R}(A_1)} I_{\scriptscriptstyle R}^0,\quad h\in J_0,
		\end{aligned}
	\end{equation} 
	and
	\begin{equation}\label{eq:mel1}
		\begin{aligned} 
			M_1(h) & = \frac{H_y^{\scriptscriptstyle R}(B)}{H_y^{\scriptscriptstyle C}(B)}I^{\scriptscriptstyle 1}_{\scriptscriptstyle C} + \frac{H_y^{\scriptscriptstyle R}(B)H_y^{\scriptscriptstyle C}(B_1)}{H_y^{\scriptscriptstyle C}(B)H_y^{\scriptscriptstyle R}(B_1)} I^{\scriptscriptstyle 1}_{\scriptscriptstyle R}, \quad h\in J_1, 
		\end{aligned}
	\end{equation}  
	where
	$$
	I^{\scriptscriptstyle 0}_{\scriptscriptstyle C}= \int_{\widehat{A_3A}}g_{\scriptscriptstyle C}dx-f_{\scriptscriptstyle C}dy,\quad I^{\scriptscriptstyle 0}_{\scriptscriptstyle L}= \int_{\widehat{A_2A_3}}g_{\scriptscriptstyle L}dx-f_{\scriptscriptstyle L}dy,\quad \bar{I}^{\scriptscriptstyle 0}_{\scriptscriptstyle C}=\int_{\widehat{A_1A_2}}g_{\scriptscriptstyle C}dx-f_{\scriptscriptstyle C}dy,
	$$
	$$
	I^{\scriptscriptstyle 0}_{\scriptscriptstyle R}=\int_{\widehat{AA_1}}g_{\scriptscriptstyle C}dx-f_{\scriptscriptstyle C}dy,\quad I^{\scriptscriptstyle 1}_{\scriptscriptstyle C}=\int_{\widehat{B_1B}}g_{\scriptscriptstyle C}dx-f_{\scriptscriptstyle C}dy\quad\text{and}\quad I^{\scriptscriptstyle 1}_{\scriptscriptstyle R}=\int_{\widehat{BB_1}}g_{\scriptscriptstyle R}dx-f_{\scriptscriptstyle R}dy,
	$$
	Furthermore, if $M_i(h)$, for $i=0,1$, has a simple zero at $h^{*}$, then for $0< \epsilon <<1$, the system \eqref{eq:01} has a unique limit cycle near $L_{h^{*}}$. 
\end{theorem}


In order to compute the zeros of the first order Melnikov functions associated to system \eqref{eq:01}, it is necessary to find the open intervals $J_0$ and $J_1$, where they are defined. For this, we have the follow proposition.

\begin{proposition}
	Consider the system $\eqref{eq:01}|_{\epsilon=0}$, with $\alpha_{\scriptscriptstyle L}=a_{\scriptscriptstyle L}$, $\alpha_{\scriptscriptstyle R}=-a_{\scriptscriptstyle R}$, $b_{\scriptscriptstyle C}=1$, $c_{\scriptscriptstyle C}=-1$ and $a_{\scriptscriptstyle C}=\alpha_{\scriptscriptstyle C}=0$, satisfying the hypotheses (Hi), $i=1,2,3$. 
	Then  $J_0=(2\sqrt{\beta_{\scriptscriptstyle C}},\tau_{\scriptscriptstyle R})$ and $J_1=(0,2\sqrt{\beta_{\scriptscriptstyle C}})$, where $\tau_{\scriptscriptstyle R}=(a_{\scriptscriptstyle R}^2-b_{\scriptscriptstyle R}\beta_{\scriptscriptstyle R}-\omega_{\scriptscriptstyle R}^2)/b_{\scriptscriptstyle R}\omega_{\scriptscriptstyle R}$ with $\omega_{\scriptscriptstyle R}=\sqrt{a^2_{\scriptscriptstyle R} + b_{\scriptscriptstyle R}c_{\scriptscriptstyle R}}$, and the periodic annulus are equivalents to one of the figures of Fig. \ref{fig:03}--\ref{fig:09}. 
\end{proposition}
\begin{proof}
	Firstly, we can note that if the saddles of the right or left subsystem from   $\eqref{eq:01}|_{\epsilon=0}$ are virtual or boundary, then we have not periodic orbits passing through the three zones.	Let  $W^u_{\scriptscriptstyle R}$ and $W^s_{\scriptscriptstyle R}$ (resp. $W^u_{\scriptscriptstyle L}$ and $W^s_{\scriptscriptstyle L}$) be the unstable and stable separatrices of the saddles of the right (resp. left) subsystems from   $\eqref{eq:01}|_{\epsilon=0}$, respectively. Denote by $P_{\scriptscriptstyle L}^{i}=W^i_{\scriptscriptstyle L}\cap \{(-1,y):y\in\mathbb{R}\}$ and $P_{\scriptscriptstyle R}^{i}=W^i_{\scriptscriptstyle R}\cap \{(1,y):y\in\mathbb{R}\}$, for $i=u,s$. After some compute, is possible to show that
	$$P_{\scriptscriptstyle L}^{u}=(-1,\tau_{\scriptscriptstyle L}),\quad P_{\scriptscriptstyle L}^{s}=(-1,-\tau_{\scriptscriptstyle L}),\quad P_{\scriptscriptstyle R}^{u}=(1,-\tau_{\scriptscriptstyle R}),\quad P_{\scriptscriptstyle R}^{s}=(1,\tau_{\scriptscriptstyle R}),$$
	where $\tau_{\scriptscriptstyle R}=(a_{\scriptscriptstyle R}^2-b_{\scriptscriptstyle R}\beta_{\scriptscriptstyle R}-\omega_{\scriptscriptstyle R}^2)/b_{\scriptscriptstyle R}\omega_{\scriptscriptstyle R}$, $\tau_{\scriptscriptstyle L}=(a_{\scriptscriptstyle L}^2+b_{\scriptscriptstyle L}\beta_{\scriptscriptstyle L}-\omega_{\scriptscriptstyle L}^2)/b_{\scriptscriptstyle L}\omega_{\scriptscriptstyle L}$, $\omega_{\scriptscriptstyle R}=\sqrt{a^2_{\scriptscriptstyle R} + b_{\scriptscriptstyle R}c_{\scriptscriptstyle R}}$ and $\omega_{\scriptscriptstyle L}=\sqrt{a^2_{\scriptscriptstyle L} + b_{\scriptscriptstyle L}c_{\scriptscriptstyle L}}$. Moreover, we have a symmetry between the points $P_{\scriptscriptstyle L}^{u}$ and $P_{\scriptscriptstyle L}^{s}$ (resp. $P_{\scriptscriptstyle R}^{u}$ and $P_{\scriptscriptstyle R}^{s}$) with respect to $x$-axis.  
	
	As the vector field $X_{\scriptscriptstyle C}$ associated with the central subsystem from $\eqref{eq:01}|_{\epsilon=0}$ is $X_{\scriptscriptstyle C}(x,y)=(y,-x+\beta_{\scriptscriptstyle C})$, with $\beta_{\scriptscriptstyle C}>0$, then we have a periodic orbit ${L}_0^1$ tangent to straight line $x=-1$ in the point $P_{\scriptscriptstyle L}=(-1,0)$ that intercept crosswise the straight line $x=1$ in two points, $P^1_{\scriptscriptstyle R}=(1,2\sqrt{\beta_{\scriptscriptstyle C}})$ and $P^2_{\scriptscriptstyle R}=(1,-2\sqrt{\beta_{\scriptscriptstyle C}})$. If $0<\beta_{\scriptscriptstyle C}< 1$, there is a periodic orbit ${L}_0^2$ tangent to straight line $x=1$ in the point $P_{\scriptscriptstyle R}=(1,0)$ contained in the region bounded by ${L}_0^1$. Moreover, the periodic orbit of central subsystem from $\eqref{eq:01}|_{\epsilon=0}$ with initial condition at the point $P_{\scriptscriptstyle R}^{u}$ intercept the straight line $x=-1$ to positive time in two points $P^2_{\scriptscriptstyle L}=(-1,-\sqrt{\tau_{\scriptscriptstyle R}^2-4\beta_{\scriptscriptstyle C}})$ and   $P^1_{\scriptscriptstyle L}=(-1,\sqrt{\tau_{\scriptscriptstyle R}^2-4\beta_{\scriptscriptstyle C}})$. Note that this periodic orbit coincides with the periodic orbit $L_0^1$ when $\beta_{\scriptscriptstyle C}=\tau_{\scriptscriptstyle R}^2/4$ (see Fig. \ref{fig:05}--\ref{fig:10}) or it do not intercept the straight line $x=-1$ when $\tau_{\scriptscriptstyle R}^2/4<\beta_{\scriptscriptstyle C}$ (see Fig. \ref{fig:06}--\ref{fig:11}). In these cases, the orbit of system  $\eqref{eq:01}|_{\epsilon=0}$ with initial condition in the point $A(h)=(1,h)$,  $h\in(0,\tau_{\scriptscriptstyle R})$, not intersect crosswise the straight line $x=-1$ to positive times, i.e. the system $\eqref{eq:01}|_{\epsilon=0}$ has no periodic orbits passing through the three zones.
	
	\begin{figure}[h]
		\begin{center}		
			\begin{overpic}[width=4.5in]{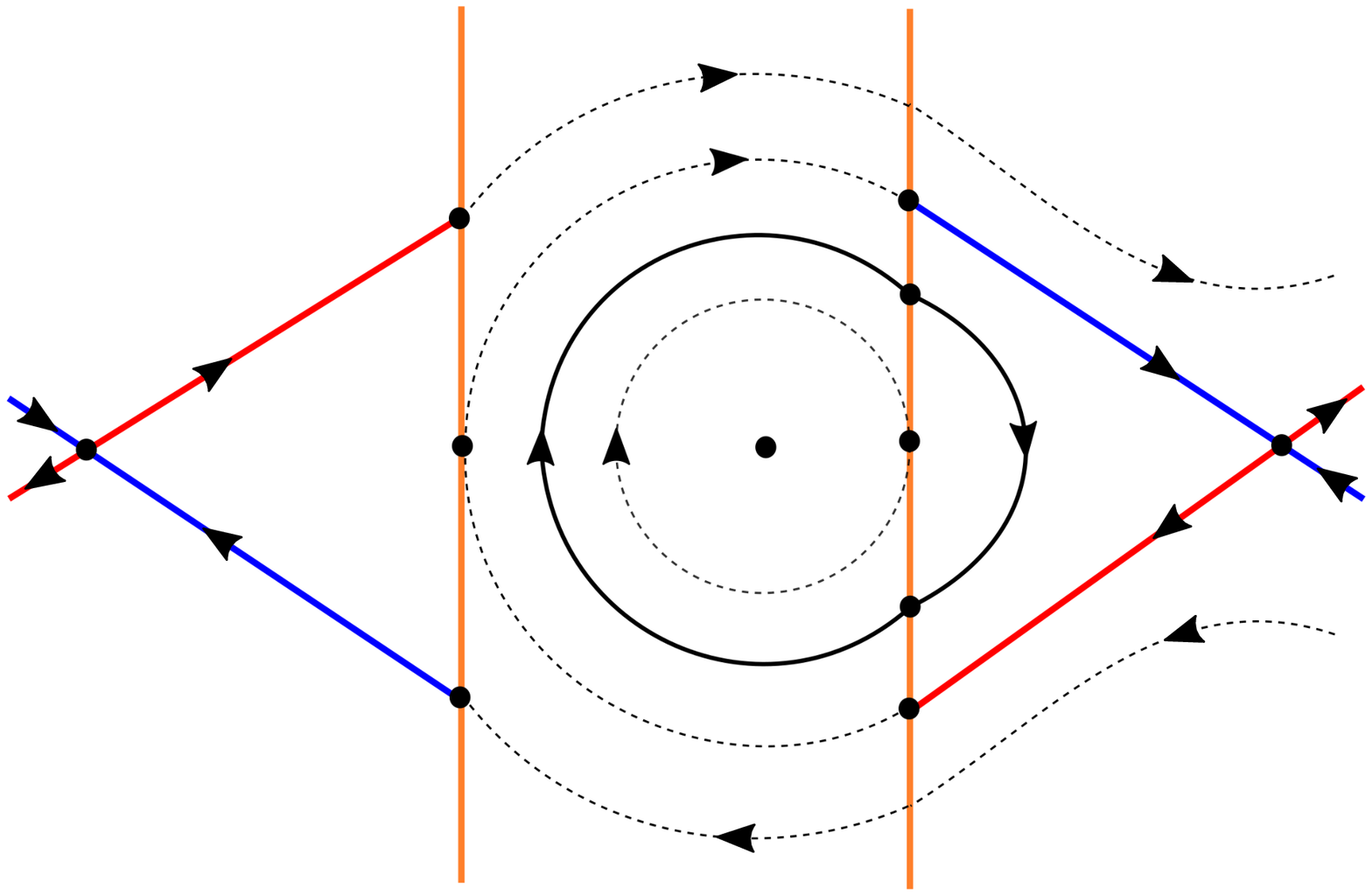}
				\put(27,-3) {$x=-1$}
				\put(62,-3) {$x=1$}
				\put(67,44) {$ B$}
				\put(67,18) {$ B_1$}
				\put(29,32) {$P_{\scriptscriptstyle L}$}
				\put(67.5,32) {$P_{\scriptscriptstyle R}$}
				\put(67,51) {$P_{\scriptscriptstyle R}^{s}$}
				\put(67,10.5) {$P_{\scriptscriptstyle R}^{u}$}
				\put(29,51) {$P_{\scriptscriptstyle L}^{u}$}
				\put(29,10) {$P_{\scriptscriptstyle L}^{s}$}
			\end{overpic}
		\end{center}
		\vspace{0.7cm}
		\caption{Phase portrait of system $\eqref{eq:01}|_{\epsilon=0}$ with  $\beta_{\scriptscriptstyle C}=\tau_{\scriptscriptstyle R}^2/4<1$.}\label{fig:05}
	\end{figure} 
	\begin{figure}[h]
		\begin{center}		
			\begin{overpic}[width=4.5in]{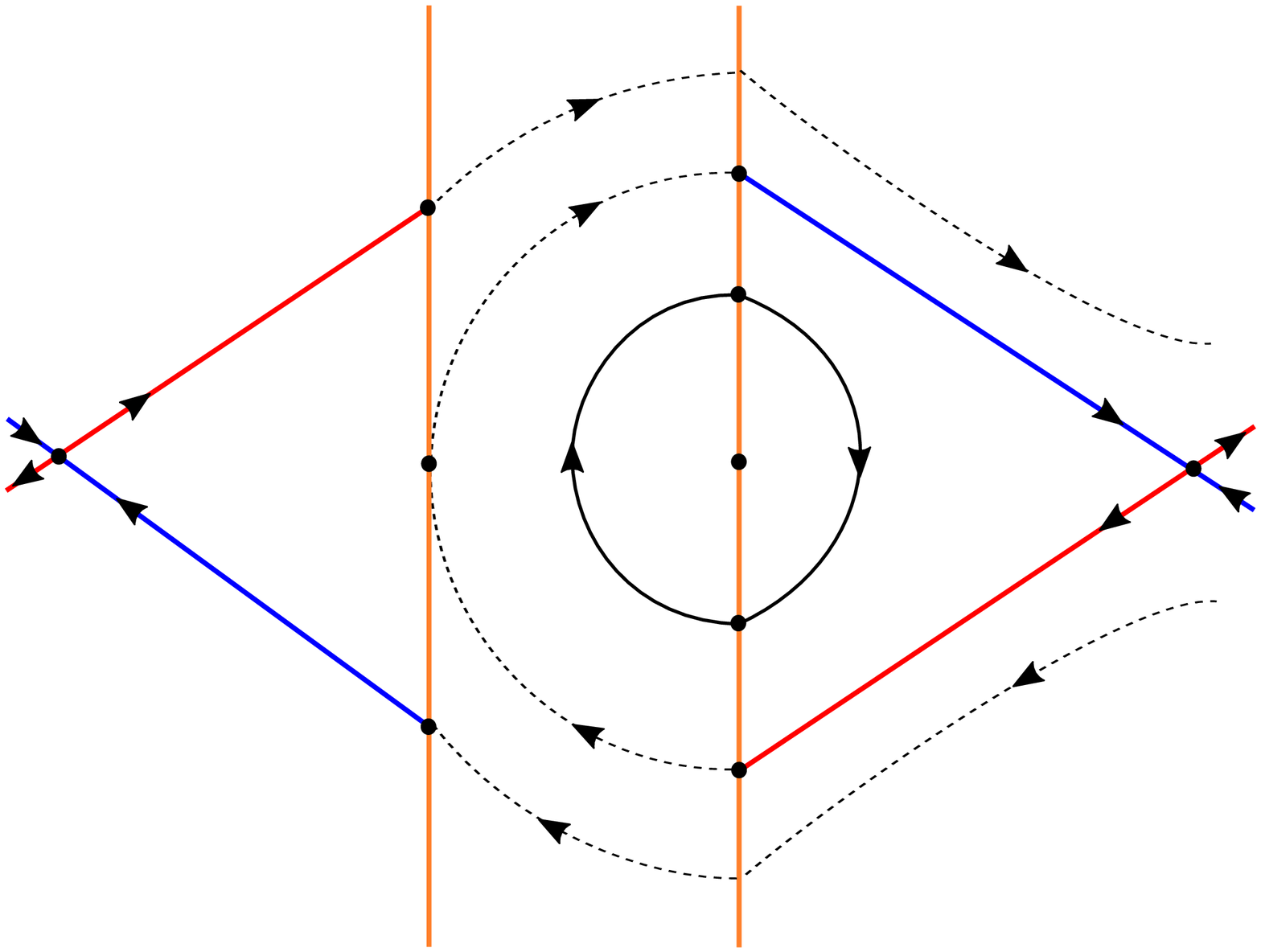}
				\put(28,-3) {$x=-1$}
				\put(55,-3) {$x=1$}
				\put(59,53) {$ B$}
				\put(59,23) {$ B_1$}
				\put(29,37) {$P_{\scriptscriptstyle L}$}
				\put(60,37) {$P_{\scriptscriptstyle R}$}
				\put(60,62) {$P_{\scriptscriptstyle R}^{s}$}
				\put(60,11) {$P_{\scriptscriptstyle R}^{u}$}
				\put(29,60) {$P_{\scriptscriptstyle L}^{u}$}
				\put(29,14) {$P_{\scriptscriptstyle L}^{s}$}
			\end{overpic}
		\end{center}
		\vspace{0.7cm}
		\caption{Phase portrait of system $\eqref{eq:01}|_{\epsilon=0}$ with  $\beta_{\scriptscriptstyle C}=\tau_{\scriptscriptstyle R}^2/4\geq 1$.}\label{fig:10}
	\end{figure} 
	\begin{figure}[h]
		\begin{center}		
			\begin{overpic}[width=4.5in]{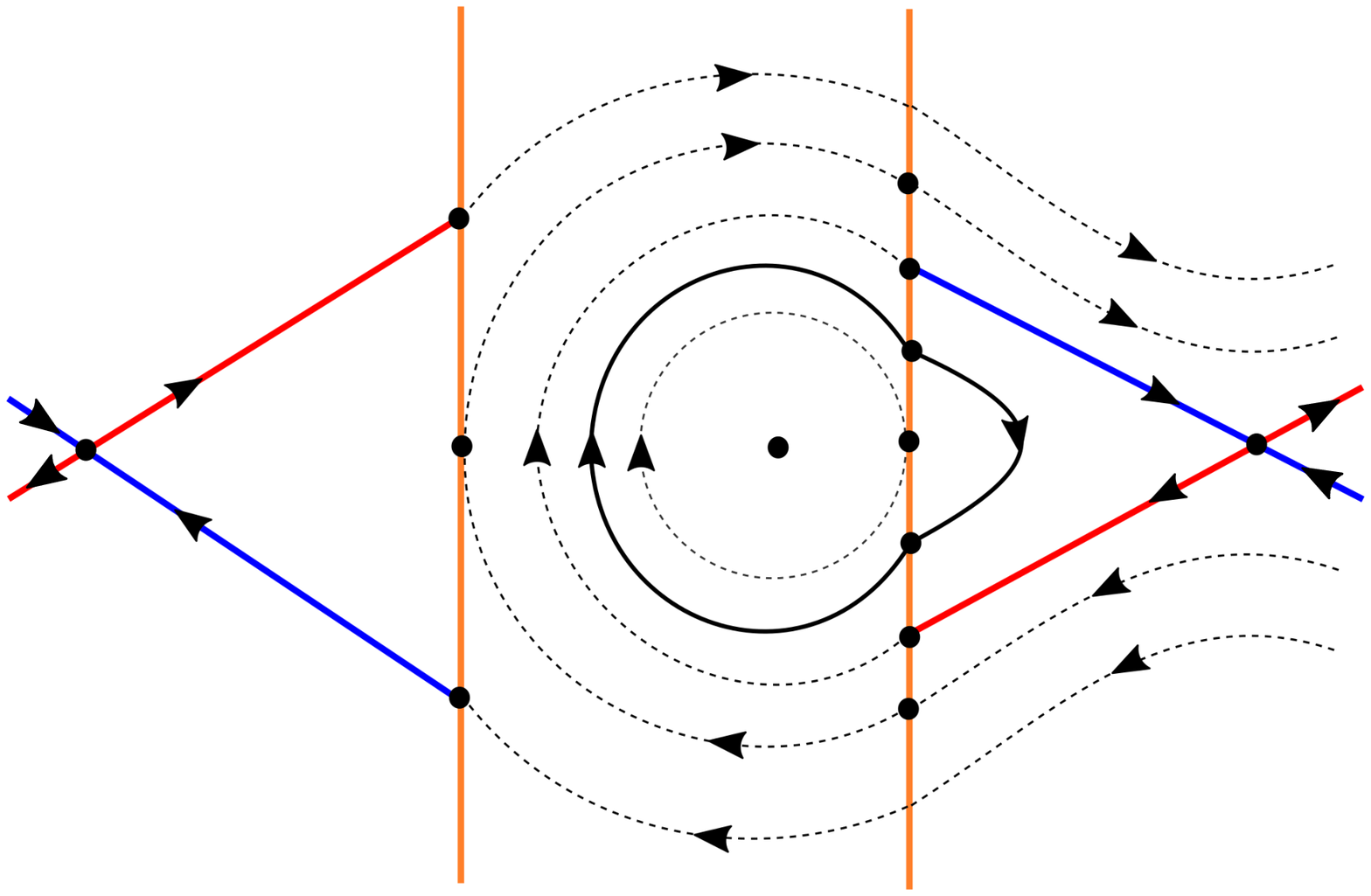}
				\put(27,-3) {$x=-1$}
				\put(62,-3) {$x=1$}
				\put(67,41) {$ B$}
				\put(67,23) {$ B_1$}
				\put(29,32) {$P_{\scriptscriptstyle L}$}
				\put(67.5,32) {$P_{\scriptscriptstyle R}$}
				\put(67,47) {$P_{\scriptscriptstyle R}^{s}$}
				\put(67,17) {$P_{\scriptscriptstyle R}^{u}$}
				\put(29,51) {$P_{\scriptscriptstyle L}^{u}$}
				\put(29,11) {$P_{\scriptscriptstyle L}^{s}$}
				\put(67.5,51.5) {$P^1_{\scriptscriptstyle R}$}
				\put(67.5,11) {$P^2_{\scriptscriptstyle R}$}
			\end{overpic}
		\end{center}
		\vspace{0.7cm}
		\caption{Phase portrait of system $\eqref{eq:01}|_{\epsilon=0}$ with  $\tau_{\scriptscriptstyle R}^2/4<\beta_{\scriptscriptstyle C}<1$.}\label{fig:06}
	\end{figure} 
	\begin{figure}[h]
		\begin{center}		
			\begin{overpic}[width=4.5in]{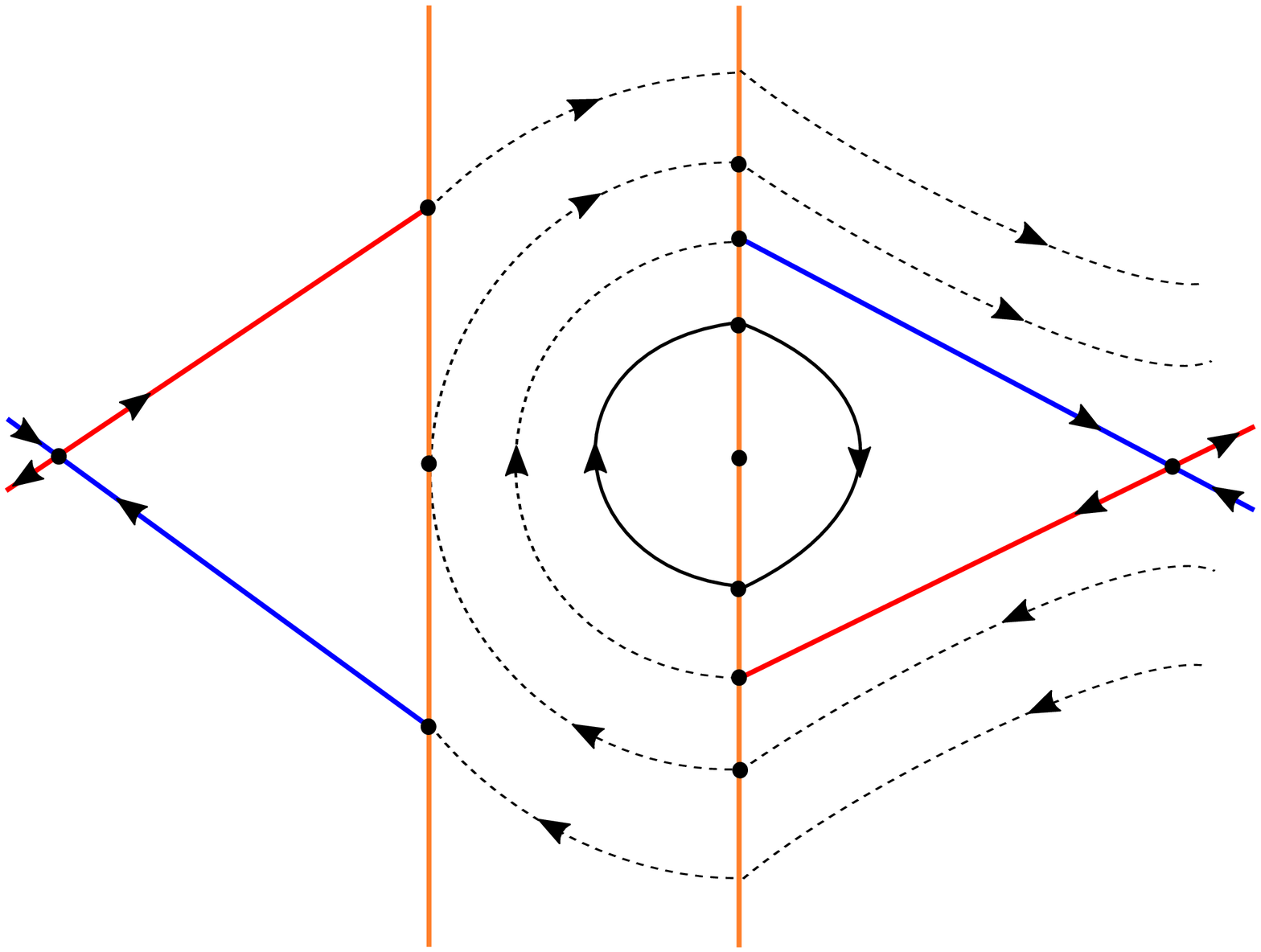}
				\put(28,-3) {$x=-1$}
				\put(55,-3) {$x=1$}
				\put(59,51) {$ B$}
				\put(59,25) {$ B_1$}
				\put(29,37) {$P_{\scriptscriptstyle L}$}
				\put(60,37) {$P_{\scriptscriptstyle R}$}
				\put(60,57) {$P_{\scriptscriptstyle R}^{s}$}
				\put(60,19) {$P_{\scriptscriptstyle R}^{u}$}
				\put(29,60) {$P_{\scriptscriptstyle L}^{u}$}
				\put(29,15) {$P_{\scriptscriptstyle L}^{s}$}
				\put(60,63) {$P^1_{\scriptscriptstyle R}$}
				\put(60,12) {$P^2_{\scriptscriptstyle R}$}
			\end{overpic}
		\end{center}
		\vspace{0.7cm}
		\caption{Phase portrait of system $\eqref{eq:01}|_{\epsilon=0}$  with  $\tau_{\scriptscriptstyle R}^2/4<\beta_{\scriptscriptstyle C}$ and $\beta_{\scriptscriptstyle C}\geq1$.}\label{fig:11}
	\end{figure} 
	
	Suppose now that $\beta_{\scriptscriptstyle C}<\tau_{\scriptscriptstyle R}^2/4$. Then, less than one reflection around the $y$-axis, we can assume without loss of generality that $\sqrt{\tau_{\scriptscriptstyle R}^2-4\beta_{\scriptscriptstyle C}}\leq\tau_{\scriptscriptstyle L}$. Moreover, if $\sqrt{\tau_{\scriptscriptstyle R}^2-4\beta_{\scriptscriptstyle C}}<\tau_{\scriptscriptstyle L}$ the system $\eqref{eq:01}|_{\epsilon=0}$ has a homoclinic loop passing through the points $P_{\scriptscriptstyle R}^{s}$, $P_{\scriptscriptstyle R}^{u}$, $P^1_{\scriptscriptstyle L}$ and $P^2_{\scriptscriptstyle L}$ (see Fig. \ref{fig:03}--\ref{fig:08}) and if $\sqrt{\tau_{\scriptscriptstyle R}^2-4\beta_{\scriptscriptstyle C}}=\tau_{\scriptscriptstyle L}$  the system $\eqref{eq:01}|_{\epsilon=0}$ has a heteroclinic orbit passing through the points $P_{\scriptscriptstyle R}^{s}$, $P_{\scriptscriptstyle R}^{u}$, $P_{\scriptscriptstyle L}^{s}$ and $P_{\scriptscriptstyle L}^{u}$ (see Fig. \ref{fig:04}--\ref{fig:09}).
	
	\begin{figure}[h]
		\begin{center}		
			\begin{overpic}[width=4.5in]{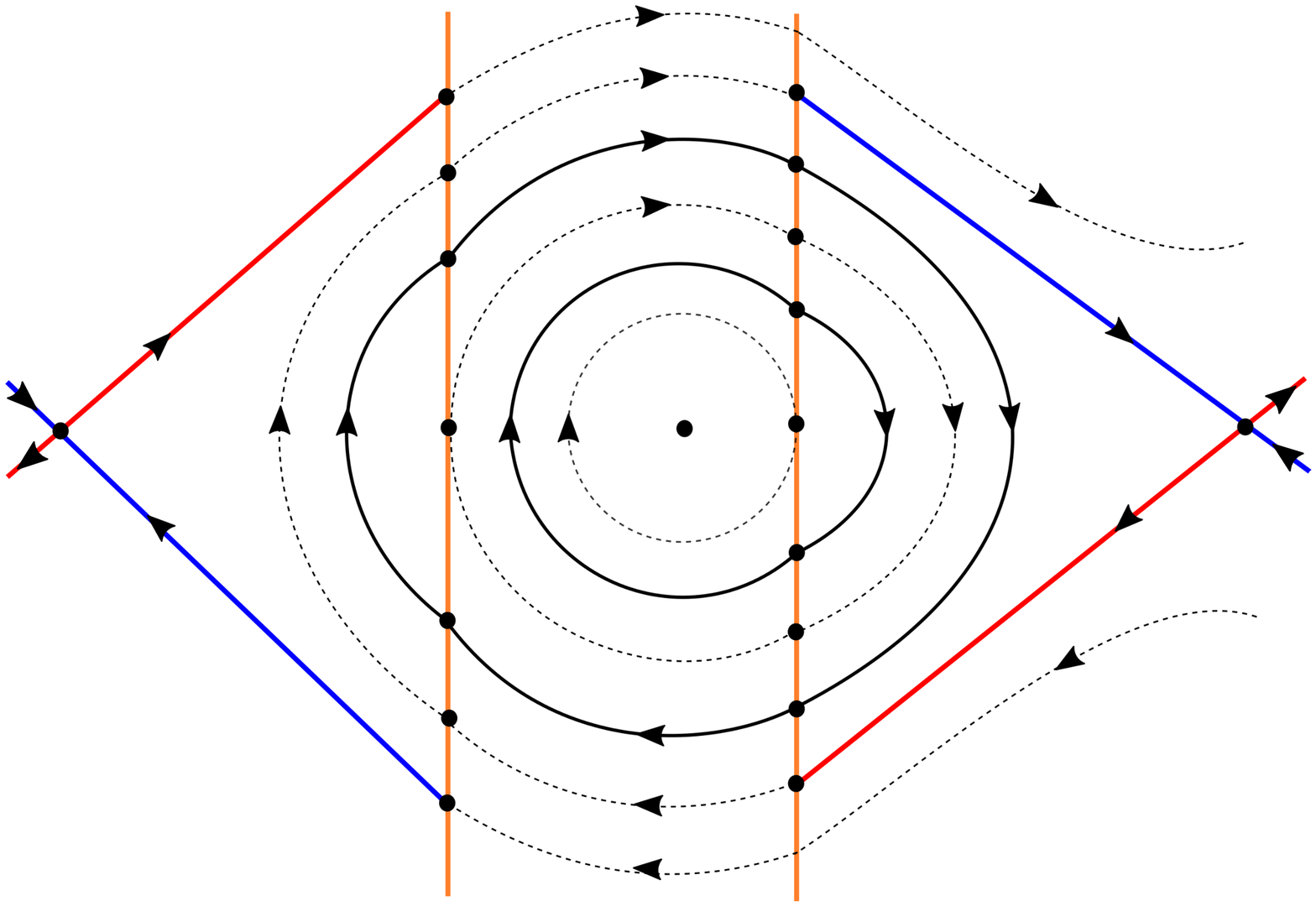}
				\put(58,-2) {$x=1$}
				\put(30,-2) {$x=-1$}
				\put(61.5,57) {$ A$}
				\put(61.5,13) {$ A_1$}
				\put(30,20) {$ A_2$}
				\put(30,50) {$ A_3$}
				\put(61.5,46) {$ B$}
				\put(62,25) {$ B_1$}
				\put(62,51) {$P^1_{\scriptscriptstyle R}$}
				\put(62,19) {$P^2_{\scriptscriptstyle R}$}
				\put(30,56.5) {$P^1_{\scriptscriptstyle L}$}
				\put(30,12) {$P^2_{\scriptscriptstyle L}$}
				\put(30,36) {$P_{\scriptscriptstyle L}$}
				\put(62.5,36) {$P_{\scriptscriptstyle R}$}
				\put(61.5,62) {$P_{\scriptscriptstyle R}^{s}$}
				\put(61.5,7.5) {$P_{\scriptscriptstyle R}^{u}$}
				\put(29.5,62) {$P_{\scriptscriptstyle L}^{u}$}
				\put(29.5,6) {$P_{\scriptscriptstyle L}^{s}$}
			\end{overpic}
		\end{center}
		\vspace{0.7cm}
		\caption{Phase portrait of system $\eqref{eq:01}|_{\epsilon=0}$ with $0<\beta_{\scriptscriptstyle C}<1$, $\beta_{\scriptscriptstyle C}<\tau_{\scriptscriptstyle R}^2/4$ and $\sqrt{\tau_{\scriptscriptstyle R}^2-4\beta_{\scriptscriptstyle C}}<\tau_{\scriptscriptstyle L}$.}\label{fig:03}
	\end{figure} 
	\begin{figure}[h]
		\begin{center}		
			\begin{overpic}[width=4.5in]{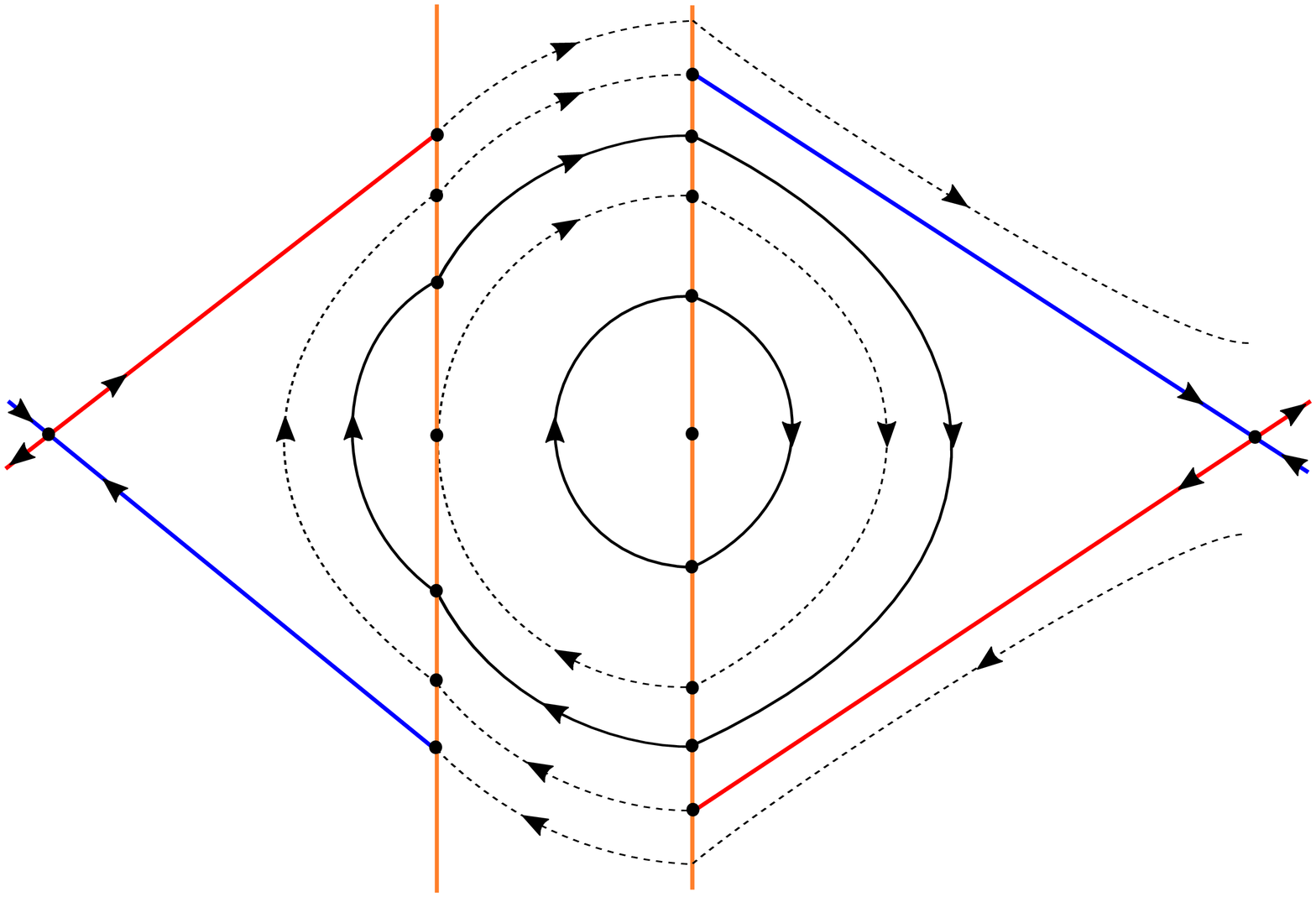}
				\put(49,-3) {$x=1$}
				\put(27,-3) {$x=-1$}
				\put(53,58) {$A$}
				\put(53,9) {$A_1$}
				\put(28,22) {$A_2$}
				\put(28,47) {$A_3$}
				\put(53,46) {$B$}
				\put(53,22) {$B_1$}
				\put(53,53) {$P^1_{\scriptscriptstyle R}$}
				\put(53,14) {$P^2_{\scriptscriptstyle R}$}
				\put(28,53) {$P^1_{\scriptscriptstyle L}$}
				\put(28,14) {$P^2_{\scriptscriptstyle L}$}
				\put(28,35) {$P_{\scriptscriptstyle L}$}
				\put(53.5,35) {$P_{\scriptscriptstyle R}$}
				\put(53,63) {$P_{\scriptscriptstyle R}^{s}$}
				\put(53,4) {$P_{\scriptscriptstyle R}^{u}$}
				\put(28,59) {$P_{\scriptscriptstyle L}^{u}$}
				\put(28,8) {$P_{\scriptscriptstyle L}^{s}$}
			\end{overpic}
		\end{center}
		\vspace{0.7cm}
		\caption{Phase portrait of system $\eqref{eq:01}|_{\epsilon=0}$ with  $1\leq\beta_{\scriptscriptstyle C}<\tau_{\scriptscriptstyle R}^2/4$ and  $\sqrt{\tau_{\scriptscriptstyle R}^2-4\beta_{\scriptscriptstyle C}}<\tau_{\scriptscriptstyle L}$.}\label{fig:08}
	\end{figure} 
	\begin{figure}[h]
		\begin{center}		
			\begin{overpic}[width=4.5in]{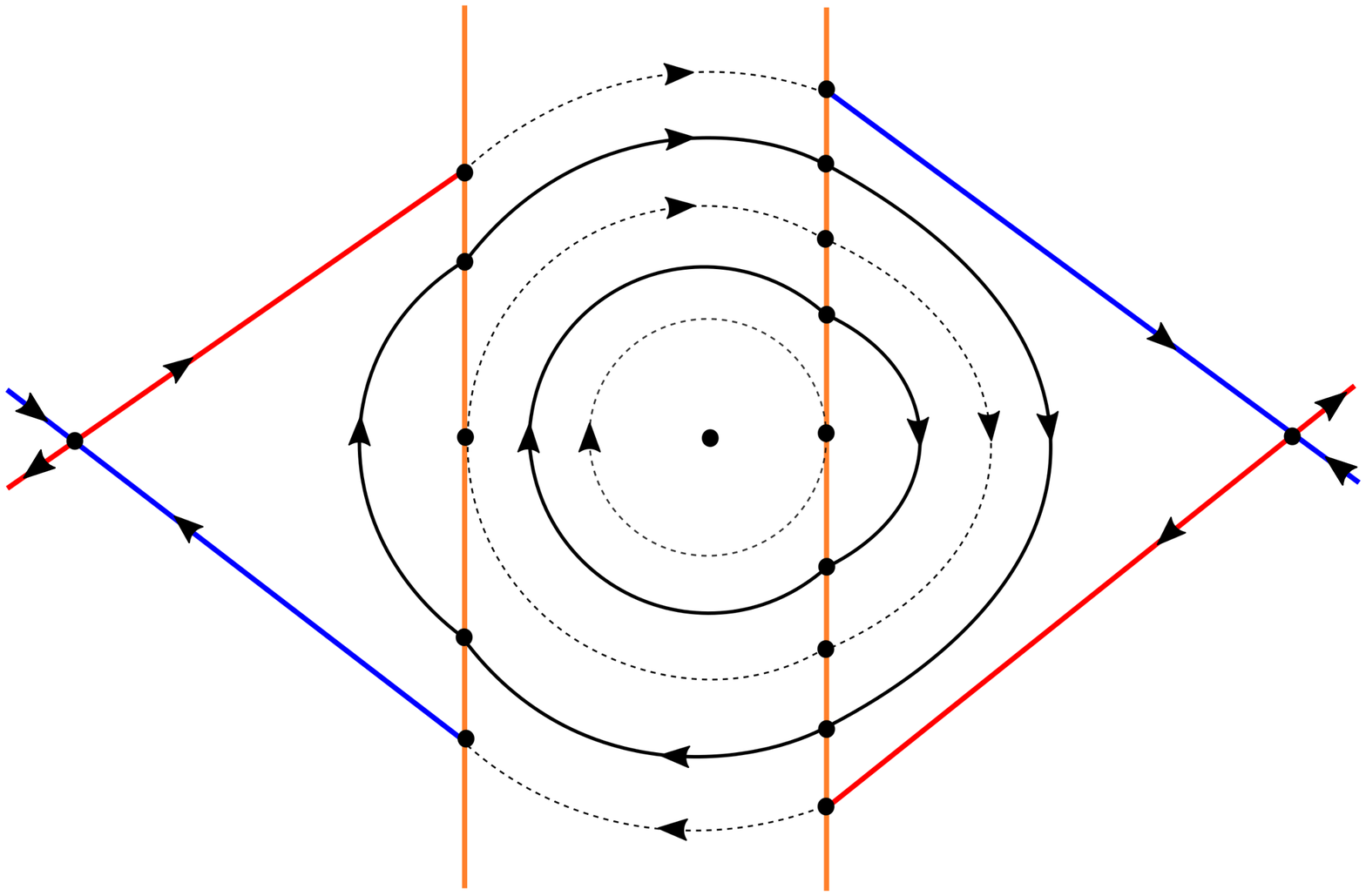}
				\put(58,-2) {$x=1$}
				\put(30,-2) {$x=-1$}
				\put(61.5,54) {$ A$}
				\put(61.5,11) {$ A_1$}
				\put(30,17) {$ A_2$}
				\put(30,48) {$ A_3$}
				\put(61.5,43) {$ B$}
				\put(61.5,22.5) {$ B_1$}
				\put(62,48.5) {$P^1_{\scriptscriptstyle R}$}
				\put(61.5,16.5) {$P^2_{\scriptscriptstyle R}$}
				\put(29.5,34) {$P_{\scriptscriptstyle L}$}
				\put(61.5,34) {$P_{\scriptscriptstyle R}$}
				\put(61.5,60) {$P_{\scriptscriptstyle R}^{s}$}
				\put(61.5,5) {$P_{\scriptscriptstyle R}^{u}$}
				\put(29.5,54) {$P_{\scriptscriptstyle L}^{u}$}
				\put(29.5,9) {$P_{\scriptscriptstyle L}^{s}$}	
			\end{overpic}
		\end{center}
		\vspace{0.7cm}
		\caption{Phase portrait of system $\eqref{eq:01}|_{\epsilon=0}$ with $0<\beta_{\scriptscriptstyle C}<1$, $\beta_{\scriptscriptstyle C}<\tau_{\scriptscriptstyle R}^2/4$ and $\sqrt{\tau_{\scriptscriptstyle R}^2-4\beta_{\scriptscriptstyle C}}=\tau_{\scriptscriptstyle L}$.}\label{fig:04}
	\end{figure} 
	\begin{figure}[h]
		\begin{center}		
			\begin{overpic}[width=4.5in]{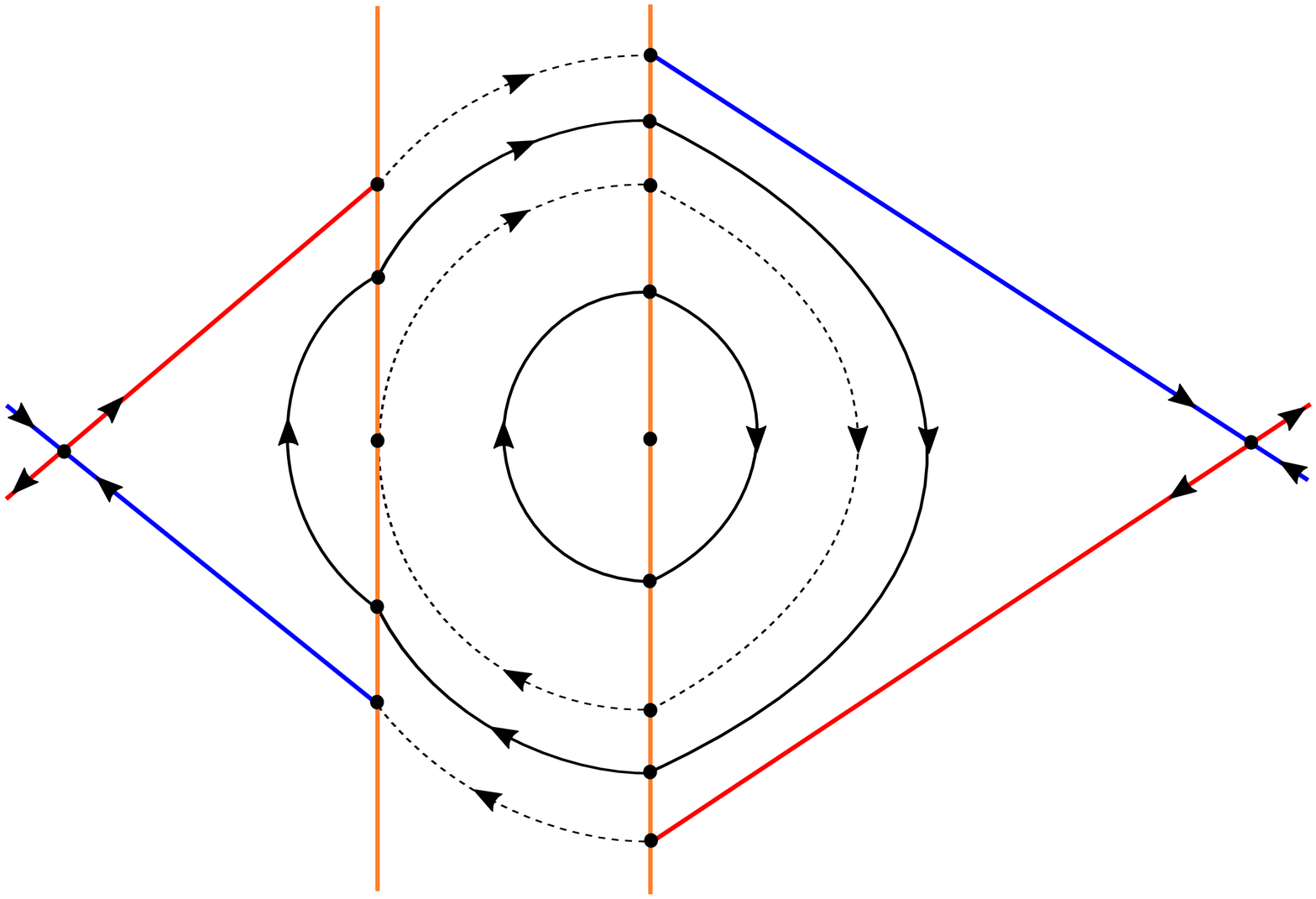}
				\put(46,-3) {$x=1$}
				\put(24,-3) {$x=-1$}
				\put(50,60) {$A$}
				\put(50,7.5) {$A_1$}
				\put(24,21) {$A_2$}
				\put(24,48) {$A_3$}
				\put(50,47) {$B$}
				\put(50,22) {$B_1$}
				\put(50,55) {$P^1_{\scriptscriptstyle R}$}
				\put(50,13) {$P^2_{\scriptscriptstyle R}$}
				\put(24,35) {$P_{\scriptscriptstyle L}$}
				\put(50.5,35) {$P_{\scriptscriptstyle R}$}
				\put(50,65) {$P_{\scriptscriptstyle R}^{s}$}
				\put(50,2) {$P_{\scriptscriptstyle R}^{u}$}
				\put(24,56) {$P_{\scriptscriptstyle L}^{u}$}
				\put(24,12) {$P_{\scriptscriptstyle L}^{s}$}
			\end{overpic}
		\end{center}
		\vspace{0.7cm}
		\caption{Phase portrait of system $\eqref{eq:01}|_{\epsilon=0}$ with  $1\leq\beta_{\scriptscriptstyle C}<\tau_{\scriptscriptstyle R}^2/4$ and $\sqrt{\tau_{\scriptscriptstyle R}^2-4\beta_{\scriptscriptstyle C}}=\tau_{\scriptscriptstyle L}$.}\label{fig:09}
	\end{figure} 
	
	Consider a initial point of the form $A(h)=(1,h)$, with $h\in (2\sqrt{\beta_{\scriptscriptstyle C}},\tau_{\scriptscriptstyle R})$. By the hypothesis (H2), the system $\eqref{eq:01}|_{\epsilon=0}$ has a family of crossing periodic orbits that intersects the straight lines $x=\pm1$ at four points, $A(h)$, $A_1(h)=(1,a_1(h))$, with $a_1(h)<h$, and $A_2(h)=(-1,a_2(h))$, $A_3(h)=(-1,a_3(h))$, with $a_2(h)<a_3(h)$ satisfying
	\begin{equation*}
		\begin{aligned}
			& H^{\scriptscriptstyle R}(A(h))=H^{\scriptscriptstyle R}(A_1(h)), \\
			& H^{\scriptscriptstyle C}(A_1(h))=H^{\scriptscriptstyle C}(A_2(h)), \\
			& H^{\scriptscriptstyle L}(A_2(h))=H^{\scriptscriptstyle L}(A_3(h)), \\
			& H^{\scriptscriptstyle C}(A_3(h))=H^{\scriptscriptstyle C}(A(h)), 
		\end{aligned}
	\end{equation*}
	where $H^{\scriptscriptstyle R}$,  $H^{\scriptscriptstyle C}$ and  $H^{\scriptscriptstyle L}$ are given by normal form from Proposition \ref{fn:01}. More precisely, we have the equations
	\begin{equation*}
		\begin{aligned}
			& \frac{b_{\scriptscriptstyle R}}{2}(h-a_1(h))(h+a_1(h))=0, \\
			& \frac{1}{2}((a_1(h)-a_2(h))(a_1(h)+a_2(h))-4\beta_{\scriptscriptstyle C})=0,
			\end{aligned}
		\end{equation*}
	\begin{equation*}
		\begin{aligned}
			& \frac{b_{\scriptscriptstyle L}}{2}(a_2(h)-a_3(h))(a_2(h)+a_3(h))=0, \\
			& \frac{1}{2}((a_3(h)-h)(a_3(h)+h)+4\beta_{\scriptscriptstyle C})=0. 
		\end{aligned}
	\end{equation*}
	As $a_1(h)<h$, $a_2(h)<a_3(h)$, $b_{\scriptscriptstyle R}>0$ and $b_{\scriptscriptstyle L}>0$, the only solution of system above is $a_1(h)=-h$, $a_2(h)=-\sqrt{h^2-4\beta_{\scriptscriptstyle C}}$ and $a_3(h)=\sqrt{h^2-4\beta_{\scriptscriptstyle C}}$, i.e. we have the four points given by $A(h)=(1,h)$, $A_1(h)=(1,-h)$, $A_2(h)=(-1,-\sqrt{h^2-4\beta_{\scriptscriptstyle C}})$ and $A_3(h)=(-1,\sqrt{h^2-4\beta_{\scriptscriptstyle C}})$. Moreover, system  $\eqref{eq:01}|_{\epsilon=0}$ has a periodic orbit $L_h$ passing through these points, for all $h\in(2\sqrt{\beta_{\scriptscriptstyle C}},\tau_{\scriptscriptstyle R})$. 
	
	Now, consider a initial point of the form $B(h)=(1,h)$, with $h\in (0,2\sqrt{\beta_{\scriptscriptstyle C}})$. By hypotheses (H2), the system $\eqref{eq:01}|_{\epsilon=0}$ has a family of crossing periodic orbits that intersects the straight line $x=1$ at two points, $B(h)$ and $B_1(h)=(1,b_1(h))$, with $b_1(h)<h$ satisfying
	\begin{equation*}
		\begin{aligned}
			& H^{\scriptscriptstyle R}(B(h))=H^{\scriptscriptstyle R}(B_1(h)), \\
			& H^{\scriptscriptstyle C}(B_1(h))=H^{\scriptscriptstyle C}(B(h)).
		\end{aligned}
	\end{equation*}
	More precisely, we have the equations 
	\begin{equation*}
		\begin{aligned}
			& \frac{b_{\scriptscriptstyle R}}{2}(h-b_1(h))(h+b_1(h))=0, \\
			& \frac{1}{2}(b_1(h)-h)(b_1(h)+h)=0.
		\end{aligned}
	\end{equation*}
	As $b_1(h)<h$ and $b_{\scriptscriptstyle R}>0$, the only solution of system above is $b_1(h)=-h$, i.e. we have the two points given by $B(h)=(1,h)$ and $B_1(h)=(1,-h)$. Moreover, system  $\eqref{eq:01}|_{\epsilon=0}$ has a periodic orbit $L_h$ passing through these points, for all $h\in(0,2\sqrt{\beta_{\scriptscriptstyle C}})$. If $h\in[\tau_{\scriptscriptstyle R},\infty)$ then the orbit of system $\eqref{eq:01}|_{\epsilon=0}$ with initial condition in $A(h)$ do not return to straight line $x=1$ to positive times, i.e. the system $\eqref{eq:01}|_{\epsilon=0}$ has no periodic orbit passing thought the point $A(h)$. Therefore, the system $\eqref{eq:01}|_{\epsilon=0}$ has a periodic annulus, formed by the periodic orbits $L_h$, for $h\in(0,2\sqrt{\beta_{\scriptscriptstyle C}})$, passing by two zones and  bounded by the periodic orbit $L_{0}^1$, when $\beta_{\scriptscriptstyle C}\geq1$, or the two periodic orbits $L_{0}^1$ and $ L_{0}^2$, when $0<\beta_{\scriptscriptstyle C}<1$. Now, assuming that $\beta_{\scriptscriptstyle C}<\tau_{\scriptscriptstyle R}^2/4$, the system $\eqref{eq:01}|_{\epsilon=0}$ has a periodic annulus, formed by the periodic orbits $L_h$, for $h\in(2\sqrt{\beta_{\scriptscriptstyle C}}, \tau_{\scriptscriptstyle R})$, passing by three zones and bounded by the periodic orbit $ L_{0}^1$ and a homoclinic loop when $\sqrt{\tau_{\scriptscriptstyle R}^2-4\beta_{\scriptscriptstyle C}}<\tau_{\scriptscriptstyle L}$ or a heteroclinic orbit when $\sqrt{\tau_{\scriptscriptstyle R}^2-4\beta_{\scriptscriptstyle C}}=\tau_{\scriptscriptstyle L}$. This complete the proof.
\end{proof}


As, for $h\in (2\sqrt{\beta_{\scriptscriptstyle C}},\tau_{\scriptscriptstyle R})$, $A(h)=(1,h)$, $A_1(h)=(1,-h)$, $A_2(h)=(-1,-\sqrt{h^2-4\beta_{\scriptscriptstyle C}})$,  $A_3(h)=(-1,\sqrt{h^2-4\beta_{\scriptscriptstyle C}})$ and, for $h\in (0,2\sqrt{\beta_{\scriptscriptstyle C}})$, $B(h)=(1,h)$, $B_1(h)=(1,-h)$, we have the follow immediate corollary. 

\begin{corollary}\label{col:mel1}
	For $J_0=(2\sqrt{\beta_{\scriptscriptstyle C}},\tau_{\scriptscriptstyle R})$ and  $J_1=(0,2\sqrt{\beta_{\scriptscriptstyle C}})$, the interval of definition of Melnikov functions given by \eqref{eq:mel0} and  \eqref{eq:mel1}, respectively, we have that
	$$
	\frac{H_y^{\scriptscriptstyle R}(A)}{H_y^{\scriptscriptstyle C}(A)}=\frac{H_y^{\scriptscriptstyle R}(B)}{H_y^{\scriptscriptstyle C}(B)}=\frac{H_y^{\scriptscriptstyle R}(A)H_y^{\scriptscriptstyle C}(A_3)H_y^{\scriptscriptstyle L}(A_2)}{H_y^{\scriptscriptstyle C}(A)H_y^{\scriptscriptstyle L}(A_3)H_y^{\scriptscriptstyle C}(A_2)}=b_{\scriptscriptstyle R},\quad \frac{H_y^{\scriptscriptstyle R}(A)H_y^{\scriptscriptstyle C}(A_3)}{H_y^{\scriptscriptstyle C}(A)H_y^{\scriptscriptstyle L}(A_3)}=\frac{b_{\scriptscriptstyle R}}{b_{\scriptscriptstyle L}},
	$$
	and
	$$\frac{H_y^{\scriptscriptstyle R}(B)H_y^{\scriptscriptstyle C}(B_1)}{H_y^{\scriptscriptstyle C}(B)H_y^{\scriptscriptstyle L}(B_1)}=\frac{H_y^{\scriptscriptstyle R}(A)H_y^{\scriptscriptstyle C}(A_3)H_y^{\scriptscriptstyle L}(A_2)H_y^{\scriptscriptstyle C}(A_1)}{H_y^{\scriptscriptstyle C}(A)H_y^{\scriptscriptstyle L}(A_3)H_y^{\scriptscriptstyle C}(A_2)H_y^{\scriptscriptstyle R}(A_1)}=1.
	$$
	Then, the first order Melnikov functions associated to system \eqref{eq:01} can be rewritten as
	\begin{equation}\label{eq:mel00}
		\begin{aligned}
			& M_0(h) = b_{\scriptscriptstyle R}\int_{\widehat{A_3A}}g_{\scriptscriptstyle C}dx-f_{\scriptscriptstyle C}dy+\frac{b_{\scriptscriptstyle R}}{b_{\scriptscriptstyle L}} \int_{\widehat{A_2A_3}}g_{\scriptscriptstyle L}dx-f_{\scriptscriptstyle L}dy \\
			& \quad\quad\quad \,\,\,\, + b_{\scriptscriptstyle R}\int_{\widehat{A_1A_2}}g_{\scriptscriptstyle C}dx-f_{\scriptscriptstyle C}dy+\int_{\widehat{AA_1}}g_{\scriptscriptstyle R}dx-f_{\scriptscriptstyle R}dy
		\end{aligned}
	\end{equation}
	and
	\begin{equation}\label{eq:mel11}
		M_1(h) = b_{\scriptscriptstyle R}\int_{\widehat{B_1B}}g_{\scriptscriptstyle C}dx-f_{\scriptscriptstyle C}dy+\int_{\widehat{BB_1}}g_{\scriptscriptstyle R}dx-f_{\scriptscriptstyle R}dy.
	\end{equation}
\end{corollary}


\section{Proof of Theorem \ref{the:01}}\label{sec:Teo}

In order to prove the Theorem \ref{the:01}, we will simplify the expression to the first order Melnikov functions associated to system \eqref{eq:01}. For this, we define the follows functions: 
\begin{equation}\label{eq:func}
	\begin{aligned}
		f_1(h) & =\sqrt{h^2-4\beta_{\scriptscriptstyle C}} ,\\	
		f_2(h) & =-\bigg(\frac{h^2+2-hf_1(h)-2\beta_{\scriptscriptstyle C}}{h^2+(\beta_{\scriptscriptstyle C}-1)^2}\bigg),\\
		f_3(h) & =\sqrt{1-\bigg(\frac{h f_1(h)-1+\beta_{\scriptscriptstyle C}^2}{h^2+(\beta_{\scriptscriptstyle C}-1)^2}\bigg)^2} ,\\	
		f_4(h) & =\bigg(\frac{hf_1(h)-1+\beta_{\scriptscriptstyle C}^2}{h^2+(\beta_{\scriptscriptstyle C}-1)^2}\bigg) ,\\
		f_0^{\scriptscriptstyle R}(h) & = (h-\tau_{\scriptscriptstyle R})(h+\tau_{\scriptscriptstyle R})\log\bigg(\frac{h+\tau_{\scriptscriptstyle R}}{\tau_{\scriptscriptstyle R}-h}\bigg),\\
		f_0^{\scriptscriptstyle L}(h) & = (h^2-4\beta_{\scriptscriptstyle C}-\tau_{\scriptscriptstyle L}^2)\log\bigg(\frac{f_1(h)+\tau_{\scriptscriptstyle L}}{\tau_{\scriptscriptstyle L}-f_1(h)}\bigg),\\
		f_0^{\scriptscriptstyle C}(h) & = (h^2+(\beta_{\scriptscriptstyle C}-1)^2)\arccos(f_4(h)),\\
		f_1^{\scriptscriptstyle C}(h) & = (h^2+(\beta_{\scriptscriptstyle C}-1)^2)\bigg(2\pi - \arccos\bigg(\frac{1-h^2-2\beta_{\scriptscriptstyle C}+\beta_{\scriptscriptstyle C}^2}{h^2+(\beta_{\scriptscriptstyle C}-1)^2}\bigg)\bigg),
	\end{aligned}
\end{equation}	
with $h\in(0,\tau_{\scriptscriptstyle R})$ and $0<\beta_{\scriptscriptstyle C}<\tau_{\scriptscriptstyle R}^2/4$.

\begin{theorem}\label{theo:melfunc}
	The first order Melnikov functions $M_0(h)$ and $M_1(h)$ given by \eqref{eq:mel00} and \eqref{eq:mel11}, respectively, can be expressed as	
	\begin{equation}\label{eq:melfunc0}
		\begin{aligned}
			M_{0}(h)& = k_0^0 h + (k_1^0 + k_1^1h)f_2(h) + (k_2^0 + k_2^1h)f_3(h) + (k_3^0 + k_3^1h + k_3^2h^2)  \\
			& \quad f_3(h)^2 + (k_4^0 + k_4^1h + k_4^2h^2)f_3(h)f_4(h) + k_5^0f_1(h) + k_6^0f_1(h)f_2(h) \\
			& \quad + k_7^0f_1(h)f_3(h)+ k_8^0f_1(h)f_3(h)^2 + k_9^0f_1(h)f_3(h)f_4(h) + k_{10}^0f_0^{\scriptscriptstyle R}(h) \\
			& \quad + k_{11}^0f_0^{\scriptscriptstyle L}(h) + k_{12}^{0}f_0^{\scriptscriptstyle C}(h),
		\end{aligned}
	\end{equation}
	with $h\in (2 \sqrt{\beta_{\scriptscriptstyle C}},\tau_{\scriptscriptstyle R})$, and
	\begin{equation}\label{eq:melfunc1}
		M_{1}(h)= k_{13}^0h + k_{14}^0f_0^{\scriptscriptstyle R}(h) + k_{15}^0f_1^{\scriptscriptstyle C}(h),
	\end{equation}
	with $h\in (0,2 \sqrt{\beta_{\scriptscriptstyle C}})$.
	The functions $f_i,f_0^{\scriptscriptstyle R},f_0^{\scriptscriptstyle L},f_j^{\scriptscriptstyle C}$, $i=1,\dots,4$ and $j=0,1$, are given in \eqref{eq:func}. Here the coefficients $k_i^j$, $i=0,\dots,15$ and $j=0,1,2$, depend on the parameters of system \eqref{eq:01}.
\end{theorem}
\begin{proof}
	Firstly, let is simplify the Melnikov function $ M_{0}(h)$ given by \eqref{eq:mel00}. For this propose, consider the orbit $(x_{\scriptscriptstyle R}(t),y_{\scriptscriptstyle R}(t))$ of system $\eqref{eq:01}|_{\epsilon=0}$, such that $(x_{\scriptscriptstyle R}(0),$ $y_{\scriptscriptstyle R}(0))=(1,h)$, given by
	\begin{equation*}
		\begin{aligned}
			x_{\scriptscriptstyle R}(t)=\,&-\frac{e^{-t \omega_{\scriptscriptstyle R}}}{2\omega_{\scriptscriptstyle R}}\Big(b_{\scriptscriptstyle R} h-b_{\scriptscriptstyle R}e^{2t \omega_{\scriptscriptstyle R}}h-2e^{t \omega_{\scriptscriptstyle R}}\omega_{\scriptscriptstyle R}+b_{\scriptscriptstyle R}\tau_{\scriptscriptstyle R}-2b_{\scriptscriptstyle R}e^{t \omega_{\scriptscriptstyle R}}\tau_{\scriptscriptstyle R} +b_{\scriptscriptstyle R}e^{2t \omega_{\scriptscriptstyle R}}\tau_{\scriptscriptstyle R}\Big), \\
			y_{\scriptscriptstyle R}(t)=\,&-\frac{e^{-t \omega_{\scriptscriptstyle R}}}{2\omega_{\scriptscriptstyle R}}\Big(-a_{\scriptscriptstyle R}h+a_{\scriptscriptstyle R}e^{2t \omega_{\scriptscriptstyle R}}h-\omega_{\scriptscriptstyle R}h- e^{2t \omega_{\scriptscriptstyle R}}\omega_{\scriptscriptstyle R}h-a_{\scriptscriptstyle R}\tau_{\scriptscriptstyle R}+2a_{\scriptscriptstyle R}e^{t \omega_{\scriptscriptstyle R}}\tau_{\scriptscriptstyle R}\\
			&-a_{\scriptscriptstyle R}e^{2t \omega_{\scriptscriptstyle R}}\tau_{\scriptscriptstyle R}-\omega_{\scriptscriptstyle R}\tau_{\scriptscriptstyle R}+e^{2t \omega_{\scriptscriptstyle R}}\omega_{\scriptscriptstyle R}\tau_{\scriptscriptstyle R}\Big).
		\end{aligned}
	\end{equation*}
	The flight time of the orbit $(x_{\scriptscriptstyle R}(t),y_{\scriptscriptstyle R}(t))$, from $A(h)=(1,h)$ to $A_1(h)=(1,-h)$, is 
	$$t_{\scriptscriptstyle R}=\frac{1}{\omega_{\scriptscriptstyle R}}\log\bigg(\frac{h+\tau_{\scriptscriptstyle R}}{\tau_{\scriptscriptstyle R}-h}\bigg).$$
	Now, for $g_{\scriptscriptstyle R}$ and $f_{\scriptscriptstyle R}$ defined in \eqref{eq:02} and \eqref{eq:03}, respectively, we have
	\begin{equation}\label{sys:r}
		\begin{aligned}
			I_{\scriptscriptstyle R}^0&=\int_{\widehat{AA_1}}g_{\scriptscriptstyle R}dx-f_{\scriptscriptstyle R}dy \\
			& =\int_{0}^{t_{\scriptscriptstyle R}}(g_{\scriptscriptstyle R}(x_{\scriptscriptstyle R}(t), y_{\scriptscriptstyle R}(t)) \dot{x}_{\scriptscriptstyle R}(t) - 
			f_{\scriptscriptstyle R}(x_{\scriptscriptstyle R}(t), y_{\scriptscriptstyle R}(t)) \dot{y}_{\scriptscriptstyle R}(t))dt \\
			& = \alpha_1h+\alpha_2f_0^{\scriptscriptstyle R}(h),		
		\end{aligned}
	\end{equation}
	with 
	$$
	\alpha_1 =\frac{1}{\omega_{\scriptscriptstyle R}}\Big(2 (p_{00} + p_{10}) \omega_{\scriptscriptstyle R} + b_{\scriptscriptstyle R} (p_{10} + q_{01}) \tau_{\scriptscriptstyle R}\Big)\quad\text{and}\quad
	\alpha_2  = \frac{b_{\scriptscriptstyle R}}{2\omega_{\scriptscriptstyle R}}\Big(p_{10} + q_{01}\Big).		
	$$
	The orbit $(x_{\scriptscriptstyle C1}(t),y_{\scriptscriptstyle C1}(t))$ of system $\eqref{eq:01}|_{\epsilon=0}$, such that $(x_{\scriptscriptstyle C1}(0),y_{\scriptscriptstyle C1}(0))=(1,-h)$, is given by
	\begin{equation*}
		\begin{aligned}
			x_{\scriptscriptstyle C1}(t)&=\beta_{\scriptscriptstyle C}+(1-\beta_{\scriptscriptstyle C})\cos(t) - h \sin(t), \\
			y_{\scriptscriptstyle C1}(t)&=-h \cos(t) + (\beta_{\scriptscriptstyle C}-1) \sin(t).
		\end{aligned}
	\end{equation*}
	The flight time of the orbit $(x_{\scriptscriptstyle C1}(t),y_{\scriptscriptstyle C1}(t))$, from $A_1(h)=(1,-h)$ to $A_2(h)=(-1,$ $-\sqrt{h^2-4\beta_{\scriptscriptstyle C}})$, is
	$$t_{\scriptscriptstyle C1}=\arccos(f_4(h)).$$
	Now, for $g_{\scriptscriptstyle C}$ and $f_{\scriptscriptstyle C}$ defined in \eqref{eq:02} and \eqref{eq:03}, respectively, we obtain 
	\begin{equation}\label{sys:c1}
		\begin{aligned}
			\bar I_{\scriptscriptstyle C}^0 = & \int_{\widehat{A_1A_2}}g_{\scriptscriptstyle C}dx-f_{\scriptscriptstyle C}dy\\
			=& \int_{0}^{t_{\scriptscriptstyle C1}}\hspace{-0.1cm}(g_{\scriptscriptstyle C}(x_{\scriptscriptstyle C1}(t), y_{\scriptscriptstyle C1}(t)) \dot{x}_{\scriptscriptstyle C1}(t)  - 
			f_{\scriptscriptstyle C}(x_{\scriptscriptstyle C1}(t), y_{\scriptscriptstyle C1}(t)) \dot{y}_{\scriptscriptstyle C1}(t))dt,\\
			=& \,\,(\alpha_{3}+\alpha_{4}h)f_2(h) + (\alpha_{5}+\alpha_{6}h+\alpha_{7}h^2) f_3(h)f_4(h) + (\alpha_{8} + \alpha_{9}h)f_3(h) \\
			& + (\alpha_{10} + \alpha_{11}h + \alpha_{12}h^2)f_3(h)^2 + \alpha_{13}f_0^{\scriptscriptstyle C}(h) ,		
		\end{aligned}
	\end{equation}
	with
	\begin{equation*}
		\begin{aligned}[c]
			&\alpha_{3}  = - (\beta_{\scriptscriptstyle C}-1 ) (v_{00} + v_{10} \beta_{\scriptscriptstyle C}), \\
			&\alpha_{4}  = u_{00} + u_{10} \beta_{\scriptscriptstyle C}, \\ 
			&\alpha_{5}  =\frac{1}{2}\Big(u_{10} - v_{01}\Big) (\beta_{\scriptscriptstyle C}-1 )^2,\\
			&\alpha_{6}  = (u_{01} + v_{10}) (\beta_{\scriptscriptstyle C}-1),\\
			&\alpha_{7}  =\frac{1}{2}\Big( v_{01}-u_{10}\Big),\\
			&\alpha_{8}  = -(u_{00} (\beta_{\scriptscriptstyle C}-1) + u_{10} (\beta_{\scriptscriptstyle C}-1 ) \beta_{\scriptscriptstyle C}),\\
		\end{aligned}
		\quad
		\begin{aligned}[c]
			&\alpha_{9}  = -(v_{00} + v_{10} \beta_{\scriptscriptstyle C}),\\
			&\alpha_{10}  = \frac{1}{2}\Big(-u_{01} - v_{10} (\beta_{\scriptscriptstyle C}-1)^2 + 2 u_{01} \beta_{\scriptscriptstyle C} - u_{01} \beta_{\scriptscriptstyle C}^2\Big),\\
			&\alpha_{11}  = (u_{10} - v_{01}) (\beta_{\scriptscriptstyle C}-1 ),\\
			&\alpha_{12}  = \frac{1}{2}\Big(u_{01} + v_{10}\Big),\\
			&\alpha_{13}  = \frac{1}{2}\Big(u_{10} + v_{01}\Big).
		\end{aligned}
	\end{equation*}
	The orbit $(x_{\scriptscriptstyle L}(t),y_{\scriptscriptstyle L}(t))$ of system $\eqref{eq:01}|_{\epsilon=0}$, such that $(x_{\scriptscriptstyle L}(0),y_{\scriptscriptstyle L}(0))=(-1,-\sqrt{h^2-4\beta_{\scriptscriptstyle C}})$, is given by
	\begin{equation*}
		\begin{aligned}
			x_{\scriptscriptstyle L}(t)=\,&\frac{e^{-t \omega_{\scriptscriptstyle L}}}{2\omega_{\scriptscriptstyle L}}\Big(b_{\scriptscriptstyle L} f_1(h)-b_{\scriptscriptstyle L}e^{2t \omega_{\scriptscriptstyle L}}f_1(h)-2e^{t \omega_{\scriptscriptstyle L}}\omega_{\scriptscriptstyle L}+b_{\scriptscriptstyle L}\tau_{\scriptscriptstyle L}-2b_{\scriptscriptstyle L}e^{t \omega_{\scriptscriptstyle L}}\tau_{\scriptscriptstyle L} +b_{\scriptscriptstyle L}e^{2t \omega_{\scriptscriptstyle L}}\tau_{\scriptscriptstyle L}\Big), \\
			y_{\scriptscriptstyle L}(t)=\,&\frac{e^{-t \omega_{\scriptscriptstyle L}}}{2\omega_{\scriptscriptstyle L}}\Big(-a_{\scriptscriptstyle L}f_1(h)+a_{\scriptscriptstyle L}e^{2t \omega_{\scriptscriptstyle L}}f_1(h)-\omega_{\scriptscriptstyle L}f_1(h)- e^{2t \omega_{\scriptscriptstyle L}}\omega_{\scriptscriptstyle L}f_1(h)-a_{\scriptscriptstyle L}\tau_{\scriptscriptstyle L}+2a_{\scriptscriptstyle L}e^{t \omega_{\scriptscriptstyle L}}\tau_{\scriptscriptstyle L}\\
			& -a_{\scriptscriptstyle L}e^{2t \omega_{\scriptscriptstyle L}}\tau_{\scriptscriptstyle L}-\omega_{\scriptscriptstyle L}\tau_{\scriptscriptstyle L}+e^{2t \omega_{\scriptscriptstyle L}}\omega_{\scriptscriptstyle L}\tau_{\scriptscriptstyle L}\Big).
		\end{aligned}
	\end{equation*}
	The flight time of the orbit $(x_{\scriptscriptstyle L}(t),y_{\scriptscriptstyle L}(t))$, from $A_2(h)=(-1,-\sqrt{h^2-4\beta_{\scriptscriptstyle C}})$ to $A_3(h)=(-1,\sqrt{h^2-4\beta_{\scriptscriptstyle C}})$, is 
	$$t_{\scriptscriptstyle L}=\frac{1}{\omega_{\scriptscriptstyle L}}\log\bigg(\frac{f_1(h)+\tau_{\scriptscriptstyle L}}{\tau_{\scriptscriptstyle L}-f_1(h)}\bigg).$$
	Now, for $g_{\scriptscriptstyle L}$ and $f_{\scriptscriptstyle L}$ defined in \eqref{eq:02} and \eqref{eq:03}, respectively, we have
	\begin{equation}\label{sys:l}
		\begin{aligned}
			I_{\scriptscriptstyle L}^0&=\int_{\widehat{A_2A_3}}g_{\scriptscriptstyle L}dx-f_{\scriptscriptstyle L}dy \\
			& =\int_{0}^{t_{\scriptscriptstyle L}}(g_{\scriptscriptstyle L}(x_{\scriptscriptstyle L}(t), y_{\scriptscriptstyle L}(t)) \dot{x}_{\scriptscriptstyle L}(t) - 
			f_{\scriptscriptstyle L}(x_{\scriptscriptstyle L}(t), y_{\scriptscriptstyle L}(t))\dot{y}_{\scriptscriptstyle L}(t))dt \\
			& = \alpha_{14}f_1(h)+\alpha_{15}f_0^{\scriptscriptstyle L}(h),		
		\end{aligned}
	\end{equation}
	with 
	$$
	\alpha_{14} =\frac{1}{\omega_{\scriptscriptstyle L}}\Big(-2 r_{00} \omega_{\scriptscriptstyle L} + 2 r_{10} \omega_{\scriptscriptstyle L} + b_{\scriptscriptstyle L} (r_{10} + s_{01}) \tau_{\scriptscriptstyle L}\Big)\quad\text{and}\quad
	\alpha_{15}  = \frac{b_{\scriptscriptstyle L}}{2\omega_{\scriptscriptstyle L}}\Big(r_{10} + s_{01}\Big).		
	$$
	The orbit $(x_{\scriptscriptstyle C2}(t),y_{\scriptscriptstyle C2}(t))$ of system $\eqref{eq:01}|_{\epsilon=0}$, such that $(x_{\scriptscriptstyle C2}(0),$ $y_{\scriptscriptstyle C2}(0))=(-1,\sqrt{h^2-4\beta_{\scriptscriptstyle C}})$, is given by
	\begin{equation*}
		\begin{aligned}
			x_{\scriptscriptstyle C2}(t)&=\beta_{\scriptscriptstyle C} - (1 + \beta_{\scriptscriptstyle C})\cos(t) + f_1(h) \sin(t), \\
			y_{\scriptscriptstyle C2}(t)&=f_1(h) \cos(t) + (1 + \beta_{\scriptscriptstyle C})\sin(t).
		\end{aligned}
	\end{equation*}
	The flight time of the orbit $(x_{\scriptscriptstyle C2}(t),y_{\scriptscriptstyle C2}(t))$, from $A_3(h)=(-1,\sqrt{h^2-4\beta_{\scriptscriptstyle C}})$ to $A(h)=(1,h)$, is
	$$t_{\scriptscriptstyle C2}=\arccos(f_4(h)).$$
	Now, for $g_{\scriptscriptstyle C}$ and $f_{\scriptscriptstyle C}$ defined in \eqref{eq:02} and \eqref{eq:03}, respectively, we obtain 
	\begin{equation}\label{sys:c2}
		\begin{aligned}
			I_{\scriptscriptstyle C}^0=&\int_{\widehat{A_3A}}g_{\scriptscriptstyle C}dx-f_{\scriptscriptstyle C}dy \\
			=&\int_{0}^{t_{\scriptscriptstyle C2}}(g_{\scriptscriptstyle C}(x_{\scriptscriptstyle C2}(t), y_{\scriptscriptstyle C2}(t))   \dot{x}_{\scriptscriptstyle C2}(t)  - 
			f_{\scriptscriptstyle C}(x_{\scriptscriptstyle C2}(t), y_{\scriptscriptstyle C2}(t)) \dot{y}_{\scriptscriptstyle C2}(t))dt \\
			= & \,\,\alpha_{16}f_2(h) + \alpha_{17}f_1(h)f_2(h) + \alpha_{18}f_3(h) + \alpha_{19}f_1(h)f_3(h)  \\
			& + (\alpha_{20} + \alpha_{21}h^2)f_3(h)^2+ \alpha_{22}f_1(h)f_3(h) + (\alpha_{23} + \alpha_{24}h^2)\\
			& \,\,f_3(h)f_4(h) + \alpha_{25}f_1(h)f_3(h)f_4(h) + \alpha_{13}f_0^{\scriptscriptstyle C}(h),		
		\end{aligned}
	\end{equation}
	with
	\begin{equation*}
		\begin{aligned}[c]
			&\alpha_{13}  = \frac{1}{2}\Big(u_{10} + v_{01}\Big),\\
			&\alpha_{16}  = - (1 + \beta_{\scriptscriptstyle C}) (v_{00} + v_{10} \beta_{\scriptscriptstyle C}), \\
			&\alpha_{17}  = -(u_{00} + u_{10} \beta_{\scriptscriptstyle C}), \\ 
			&\alpha_{18}  = -(1 + \beta_{\scriptscriptstyle C}) (u_{00} + u_{10} \beta_{\scriptscriptstyle C}),\\
			&\alpha_{19}  = (v_{00} + v_{10} \beta_{\scriptscriptstyle C}),\\
			&\alpha_{20}  =\frac{1}{2}\Big(u_{01} + v_{10}\Big)(-1 - 6 \beta_{\scriptscriptstyle C} - \beta_{\scriptscriptstyle C}^2),
		\end{aligned}
		\quad\quad\quad\quad
		\begin{aligned}[c]
			&\alpha_{21}  =\frac{1}{2}\Big(u_{01} + v_{10}\Big),\\
			&\alpha_{22}  = -(u_{10} - v_{01}) (1 + \beta_{\scriptscriptstyle C}),\\
			&\alpha_{23}  =-\frac{1}{2}\Big(u_{10} - v_{01}\Big)(-1 - 6 \beta_{\scriptscriptstyle C} - \beta_{\scriptscriptstyle C}^2),\\
			&\alpha_{24}  =-\frac{1}{2}\Big(u_{10} - v_{01}\Big),\\
			&\alpha_{25}  = - (u_{01} + v_{10}) (1 + \beta_{\scriptscriptstyle C}).
		\end{aligned}
	\end{equation*}
	Hence, by Corollary \ref{col:mel1}, the first order Melnivov function associated to system \eqref{eq:01} is given by
	\begin{equation}\label{eq:meln00}
		M_{0}(h) = b_{\scriptscriptstyle R} I^{0}_{\scriptscriptstyle C}+\frac{b_{\scriptscriptstyle R}}{b_{\scriptscriptstyle L}}I^{0}_{\scriptscriptstyle L}+b_{\scriptscriptstyle R} \bar I^{0}_{\scriptscriptstyle C}+I^{0}_{\scriptscriptstyle R}. 
	\end{equation}
	Replacing \eqref{sys:r}, \eqref{sys:c1}, \eqref{sys:l} and \eqref{sys:c2} in \eqref{eq:meln00} we obtain \eqref{eq:melfunc0}, with
	\begin{equation*}
		\begin{aligned}[c]
			&k_{0}^{0}  = \alpha_{1}, \\
			&k_{1}^{0}  = b_{\scriptscriptstyle R} (\alpha_{3} + \alpha_{16}), \\
			&k_{1}^{1}  = b_{\scriptscriptstyle R} \alpha_{4}, \\
			&k_{2}^{0}  = b_{\scriptscriptstyle R} (\alpha_{8} + \alpha_{18}), \\
			&k_{2}^{1}  = b_{\scriptscriptstyle R} \alpha_{9}, \\
			&k_{3}^{0}  = b_{\scriptscriptstyle R} (\alpha_{10} + \alpha_{20}),\\
			&k_{3}^{1}  = b_{\scriptscriptstyle R} \alpha_{11}, \\
		\end{aligned}
		\quad\quad\quad
		\begin{aligned}[c]
			&k_{3}^{2}  = b_{\scriptscriptstyle R} (\alpha_{12} + \alpha_{21}), \\
			&k_{4}^{0}  = b_{\scriptscriptstyle R} (\alpha_{5} + \alpha_{23}), \\
			&k_{4}^{1}  = b_{\scriptscriptstyle R} \alpha_{6},\\
			&k_{4}^{2}  = b_{\scriptscriptstyle R} (\alpha_{7} + \alpha_{24}), \\
			&k_{5}^{0}  = \frac{b_{\scriptscriptstyle R}}{b_{\scriptscriptstyle L}} \alpha_{14},\\
			&k_{6}^{0}  = b_{\scriptscriptstyle R} \alpha_{17}, \\	
		\end{aligned}
		\quad\quad\quad
		\begin{aligned}[c]
			&k_{7}^{0}  = b_{\scriptscriptstyle R} \alpha_{19}, \\	
			&k_{8}^{0}  = b_{\scriptscriptstyle R} \alpha_{22}, \\
			&k_{9}^{0}  = b_{\scriptscriptstyle R} \alpha_{25}, \\
			&k_{10}^{0}  =  \alpha_{2}, \\
			&k_{11}^{0}  = \frac{b_{\scriptscriptstyle R}}{b_{\scriptscriptstyle L}} \alpha_{15}, \\
			&k_{12}^{0}  = 2b_{\scriptscriptstyle R} \alpha_{13}.
		\end{aligned}
	\end{equation*}
	Now, let is simplify the Melnikov function $M_1(h)$ given by \eqref{eq:mel11}. Similarly as in previous case, we have that
	\begin{equation}\label{sys:r11}
		\begin{aligned}
			I_{\scriptscriptstyle R}^1&=\int_{\widehat{BB_1}}g_{\scriptscriptstyle R}dx-f_{\scriptscriptstyle R}dy \\
			& =\int_{0}^{t_{\scriptscriptstyle R}}(g_{\scriptscriptstyle R}(x_{\scriptscriptstyle R}(t), y_{\scriptscriptstyle R}(t)) \dot{x}_{\scriptscriptstyle R}(t) - 
			f_{\scriptscriptstyle R}(x_{\scriptscriptstyle R}(t), y_{\scriptscriptstyle R}(t)) \dot{y}_{\scriptscriptstyle R}(t))dt \\
			& = \alpha_1h+\alpha_2f_0^{\scriptscriptstyle R}(h).		
		\end{aligned}
	\end{equation}
	The orbit $(x_{\scriptscriptstyle C}(t),y_{\scriptscriptstyle C}(t))$ of system $\eqref{eq:01}|_{\epsilon=0}$, such that $(x_{\scriptscriptstyle C}(0),y_{\scriptscriptstyle C}(0))=(1,-h)$, is given by
	\begin{equation*}
		\begin{aligned}
			x_{\scriptscriptstyle C}(t)&=\beta_{\scriptscriptstyle C}+(1-\beta_{\scriptscriptstyle C})\cos(t) - h \sin(t), \\
			y_{\scriptscriptstyle C}(t)&=-h \cos(t) + (\beta_{\scriptscriptstyle C}-1) \sin(t).
		\end{aligned}
	\end{equation*}
	The flight time of the orbit $(x_{\scriptscriptstyle C}(t),y_{\scriptscriptstyle C}(t))$, from $B_1(h)=(1,-h)$ to $B(h)=(1,h)$, is
	$$t_{\scriptscriptstyle C}=2\pi - \arccos\bigg(\frac{1-h^2-2\beta_{\scriptscriptstyle C}+\beta_{\scriptscriptstyle C}^2}{h^2+1-2\beta_{\scriptscriptstyle C}+\beta_{\scriptscriptstyle C}^2}\bigg).$$
	Now, for $g_{\scriptscriptstyle C}$ and $f_{\scriptscriptstyle C}$ defined in \eqref{eq:02} and \eqref{eq:03}, respectively, we obtain 
	\begin{equation}\label{sys:c111}
		\begin{aligned}
			I_{\scriptscriptstyle C}^1 = & \int_{\widehat{B_1B}}g_{\scriptscriptstyle C}dx-f_{\scriptscriptstyle C}dy\\
			=& \int_{0}^{t_{\scriptscriptstyle C}}\hspace{-0.1cm}(g_{\scriptscriptstyle C}(x_{\scriptscriptstyle C1}(t), y_{\scriptscriptstyle C1}(t)) \dot{x}_{\scriptscriptstyle C1}(t)  - 
			f_{\scriptscriptstyle C}(x_{\scriptscriptstyle C1}(t), y_{\scriptscriptstyle C1}(t)) \dot{y}_{\scriptscriptstyle C1}(t))dt,\\
			=& \,\,\alpha_{26}h + \alpha_{27}f_1^{\scriptscriptstyle C}(h) ,		
		\end{aligned}
	\end{equation}
	with
	$$\alpha_{26}=2 u_{00} + u_{10} - v_{01} + (u_{10} + v_{01}) \beta_{\scriptscriptstyle C}\quad\text{and}\quad \alpha_{27}=\frac{1}{2}(u_{10} + v_{01}).$$
	Hence, by Corollary \ref{col:mel1}, the first order Melnivov function associated to system \eqref{eq:01} is given by
	\begin{equation}\label{eq:meln11}
		M_{1}(h) = b_{\scriptscriptstyle R} I^{1}_{\scriptscriptstyle C}+I^{1}_{\scriptscriptstyle R}. 
	\end{equation}
	Replacing \eqref{sys:r11} and \eqref{sys:c111} in \eqref{eq:meln11} we obtain \eqref{eq:melfunc1}, with
	\begin{equation*}
		k_{13}^{0}=\alpha_{1} -  b_{\scriptscriptstyle R} \alpha_{26},\quad k_{14}^{0}=\alpha_2\quad\text{and}\quad k_{15}^{0}= b_{\scriptscriptstyle R} \alpha_{27}.
	\end{equation*}
\end{proof}

\medskip


Before proving the Theorem \ref{the:01}, we will need the follow result.

\medskip

Consider the functions $F:\mathbb{R}\rightarrow\mathbb{R}$ and  $G:\mathbb{R}\rightarrow\mathbb{R}$ given by 
\begin{equation}\label{eq:m1}
	F(h)=\sum_{j=0}^{n}\Big(C_j(\delta) (h-\tau)^j+D_{2j+1}(\delta) (h-\tau)^{\frac{2j+1}{2}}\Big)+\mathcal{O}((h-\tau)^{n+1}),
\end{equation}
and
\begin{equation}\label{eq:m2}
	G(h)=\sum_{j=0}^{n}C_j(\delta) (h-\tau)^j+\mathcal{O}((h-\tau)^{n+1}),
\end{equation}
with $\tau\geq0$ and the coefficients  $C_j(\delta)$ and $D_{2j+1}(\delta)$, $j=0,\dots,n$, depending on the parameters $\delta=(\delta_1,\dots,\delta_m)\in\mathbb{R}^m$. Then we have the follow proposition.

\begin{proposition}\label{the:coef}  Suppose that there exist an integer $k\geq 1$ and $\tilde{\delta}\in\mathbb{R}^{m}$ with $m\geq 2k+1$ such that 
	\begin{equation}\label{cond:rank1}
		C_{j}(\tilde{\delta})=D_{2j+1}(\tilde{\delta})=0, \quad j=0,\dots,k-1,
	\end{equation}
and $$\textnormal{rank}\,\frac{\partial(C_{0},D_{1}\dots,D_{2k-1},C_{k})}{\partial(\delta_1,\delta_2,\dots,\delta_{m-1},\delta_m)}\tilde{(\delta)}=2k+1$$
	or
	\begin{equation}\label{cond:rank2}
		C_{j}(\tilde{\delta})=D_{2j+1}(\tilde{\delta})=C_{k-1}(\tilde{\delta})=0, \quad j=0,\dots,k-2,
	\end{equation}
and
$$\textnormal{rank}\,\frac{\partial(C_{0},D_{1}\dots,C_{k-1},D_{2k-1})}{\partial(\delta_1,\delta_2,\dots,\delta_{m-1},\delta_{m})}\tilde{(\delta)}=2k.$$
	Then, if condition \eqref{cond:rank1} holds, the functions \eqref{eq:m1} and \eqref{eq:m2} have at least $2k$ and $k$ real  zeros, respectively, in a neighborhood of $h=\tau$ for all $\delta$ near $\tilde{\delta}$. Moreover, the function \eqref{eq:m2} has exactly $k$ real simple zeros, in a neighborhood of $h=\tau$ for all $\delta$ near $\tilde{\delta}$, when $C_{k}(\tilde{\delta})\ne 0$. Now, if condition \eqref{cond:rank2} holds, the functions \eqref{eq:m1} and \eqref{eq:m2} have at least $2k-1$ and $k-1$ real  zeros, respectively, in a neighborhood of $h=\tau$ for all $\delta$ near $\tilde{\delta}$. 
\end{proposition}
\begin{proof}
	We will prove only the case where condition \eqref{cond:rank1} holds. The proof is similar when we have the condition \eqref{cond:rank2}. By the condition \eqref{cond:rank1} we can assume that
	$$\textnormal{det}\,\frac{\partial(C_{0},D_{1}\dots,D_{2k-1},C_{k})}{\partial(\delta_1,\delta_2,\dots,\delta_{2k},\delta_{2k+1})}(\tilde{\delta})\ne0,$$
	Then the change of parameters $\tilde{C}_j=C_j(\delta_1,\dots,\delta_{2k+1},\tilde{\delta}_{2k+2},\dots,\tilde{\delta}_{m})$, $j=0,\dots,k$, and $\tilde{D}_{2j+1}=D_{2j+1}(\delta_1,\dots,\delta_{2k+1},\tilde{\delta}_{2k+2},\dots,\tilde{\delta}_{m})$, $j=0,\dots,k-1$, has inverse $\delta_i(\tilde{C}_{0},\tilde{D}_{1},\dots,$ $\tilde{D}_{2k-1},\tilde{C}_{k})$, $i=1,\dots,2k+1$, and we can rewrite \eqref{eq:m1} and \eqref{eq:m2} as
	\begin{equation*}
		\begin{aligned}
			& F(h)=\tilde{C}_0+\tilde{D}_1(h-\tau)^{\frac{1}{2}}+\dots+\tilde{D}_{2k-1}(h-\tau)^{\frac{2k-1}{2}}+\tilde{C}_k(h-\tau)^k+\mathcal{O}((h-\tau)^{k+1}), \\
			& G(h) = \tilde{C}_0+\tilde{C}_1(h-\tau)+\dots+\tilde{C}_k(h-\tau)^k+\mathcal{O}((h-\tau)^{k+1})
		\end{aligned}
	\end{equation*}
	with $\tilde{C}_j(\tilde{\delta})=\tilde{D}_{2j+1}(\tilde{\delta})=0$, $j=0,\dots,k-1$.	
	
	Suppose that $k$ is even (if $k$ is odd the proof is similar) and let $0<|\tilde{C}_k-C_k(\tilde{\delta})|\ll1$, $0<h_{1}^{F}-\tau\ll1$ and $0<\tau-h_{1}^{G}\ll1$ such that $$\tilde{C}_k(h_{1}^{F}-\tau)^k>0\quad\text{and}\quad\tilde{C}_k(h_{1}^{G}-\tau)^k>0.$$ 	
	Take $\tilde{D}_{2k-1}$ such that $|\tilde{D}_{2k-1}|\ll\tilde{C}_{k}$, $\tilde{D}_{2k-1}\tilde{C}_{k}<0$ and 
	$$\tilde{D}_{2k-1}(h_{1}^{F}-\tau)^{\frac{2k-1}{2}}+\tilde{C}_k(h_{1}^{F}-\tau)^k>0.$$ 	
	As $\tilde{D}_{2k-1}<0$, we can choose $h_{2}^{F}$, such that 
	$0<h_{2}^{F}-\tau\ll h_{1}^{F}-\tau\ll 1$ and 
	$$\tilde{D}_{2k-1}(h_{2}^{F}-\tau)^{\frac{2k-1}{2}}+\tilde{C}_k(h_{2}^{F}-\tau)^k<0.$$
	Therefore, the equation   
	$$\tilde{D}_{2k-1}(h-\tau)^{\frac{2k-1}{2}}+\tilde{C}_k(h-\tau)^k=0$$
	has a zero $\tilde{h}_{1}^{F}$, with $h_{2}^{F}<\tilde{h}_{1}^{F}<h_{1}^{F}$. Now, take $\tilde{C}_{k-1}$ such that  $|\tilde{C}_{k-1}|\ll|\tilde{D}_{2k-1}|\ll\tilde{C}_{k}$, $\tilde{C}_{k-1}\tilde{D}_{2k-1}<0$,
	$$\tilde{C}_{k-1}(h_{2}^{F}-\tau)^{k-1}+\tilde{D}_{2k-1}(h_{2}^{F}-\tau)^{\frac{2k-1}{2}}+\tilde{C}_k(h_{2}^{F}-\tau)^k<0$$
	and
	$$\tilde{C}_{k-1}(h_{1}^{G}-\tau)^{k-1}+\tilde{C}_k(h_{1}^{G}-\tau)^k>0.$$
	As $\tilde{C}_{k-1}>0$, we can choose  $h_{3}^{F}$ and $h_{2}^{G}$, such that
	$0<h_{3}^{F}-\tau\ll h_{2}^{F}-\tau\ll 1$, $0<\tau-h_{2}^{G}\ll \tau-h_{1}^{G}\ll 1$,
	$$\tilde{C}_{k-1}(h_{3}^{F}-\tau)^{k-1}+\tilde{D}_{2k-1}(h_{3}^{F}-\tau)^{\frac{2k-1}{2}}+\tilde{C}_k(h_{3}^{F}-\tau)^k>0$$
	and
	$$\tilde{C}_{k-1}(h_{2}^{G}-\tau)^{k-1}+\tilde{C}_k(h_{2}^{G}-\tau)^k<0.$$
	Therefore, the equation 
	$$\tilde{C}_{k-1}(h-\tau)^{k-1}+\tilde{D}_{2k-1}(h-\tau)^{\frac{2k-1}{2}}+\tilde{C}_k(h-\tau)^k=0$$
	has one more zero $\tilde{h}_{2}^{F}$, with $h_{3}^{F}<\tilde{h}_{2}^{F}<h_{2}^{F}<\tilde{h}_{1}^{F}<h_{1}^{F}$, and the equation
	$$\tilde{C}_{k-1}(h-\tau)^{k-1}+\tilde{C}_k(h-\tau)^k=0$$
	has a zero $\tilde{h}_{1}^{G}$, with $h_{1}^{G}<\tilde{h}_{1}^{G}<h_{2}^{G}$. Continuing with this reasoning,  there are  $\tilde{C}_0,\tilde{D}_1,\dots,\tilde{D}_{2k-1},\tilde{C}_{k}$, such that 
	$$\tilde{C}_0\tilde{D}_{1}<0,\,\,\tilde{D}_1\tilde{C}_{2}<0,\dots,\,\,\tilde{D}_{2k-1}\tilde{C}_{k}<0,\quad |\tilde{C}_0|\ll|\tilde{D}_{1}|\ll\dots\ll|\tilde{D}_{2k-1}|\ll\tilde{C}_{k},$$
	and the equations \eqref{eq:m1} and \eqref{eq:m2} have $2k$ and $k$ real zeros near $\tau$, respectively. Moreover, if $\tilde{C}_k(\tilde{\delta})\ne0$ and  $\tilde{C}_j(\tilde{\delta})=\tilde{D}_{2j+1}(\tilde{\delta})=0$, $j=0,\dots,k-1$, by Rolle’s theorem, for all $\delta$ near $\tilde{\delta}$ we can choose the $\tilde{C}_i$ and $h_{j}^{G}$, with $i=0,\dots,k$ and $j=1,\dots,k$, such that equation \eqref{eq:m2} have exactly $k$ real simple zero near $h=\tau$.
\end{proof}


\begin{proof}[Proof of Theorem \ref{the:01}]
	In order to prove the Theorem \ref{the:01}, let us consider the Melnikov functions associated to system \eqref{eq:01} with $f(h)$ and $g(h)$ given by \eqref{eq:02} and \eqref{eq:03}, respectively, and $H(x,y)$ with $\alpha_{\scriptscriptstyle L}=a_{\scriptscriptstyle L}$, $\alpha_{\scriptscriptstyle R}=-a_{\scriptscriptstyle R}$, $b_{\scriptscriptstyle C}=1$, $c_{\scriptscriptstyle C}=-1$, $a_{\scriptscriptstyle C}=\alpha_{\scriptscriptstyle C}=0$. 	
	Our study will be concentrated in a neighborhood of the orbit $L_h$ with $h=2\sqrt{\beta_{\scriptscriptstyle C}}$ (i.e. the orbit $L^{1}_{0}$), since to estimate the zeros of the Melnikov functions we consider their expansions at the point corresponding to this orbit. More precisely, consider the Melnikov functions $M_{0}(h)$ and $M_{1}(h)$ given by the Theorem \ref{theo:melfunc}. Note that, the function $M_{0}(h)$ is not differentiable at $h=2\sqrt{\beta_{\scriptscriptstyle C}}$. Thus, for obtain the expansion of this function at $h=2\sqrt{\beta_{\scriptscriptstyle C}}$ we need doing a change of variables $h=u^2+2\sqrt{\beta_{\scriptscriptstyle C}}$ and, after compute its Taylor's expansion at $u=0$, we return to the original variable doing the change $u=\sqrt{h-2\sqrt{\beta_{\scriptscriptstyle C}}}$. Then the expansion of function $M_{0}(h)$ and $M_{1}(h)$ at $h=2\sqrt{\beta_{\scriptscriptstyle C}}$ are given by
	\begin{eqnarray}\label{eq:m0h}
		&&M_{0}(h)=C_{00} + D_{1}\big(h-2\sqrt{\beta_{\scriptscriptstyle C}}\big)^{\frac{1}{2}} + C_{10}\big(h-2\sqrt{\beta_{\scriptscriptstyle C}}\big) + D_{3}\big(h-2\sqrt{\beta_{\scriptscriptstyle C}}\big)^{\frac{3}{2}} \nonumber\\
		&&\quad\quad\quad \quad   + \, C_{20}\big(h-2\sqrt{\beta_{\scriptscriptstyle C}}\big)^2 +  \mathcal{O}\big(\big(h-2\sqrt{\beta_{\scriptscriptstyle C}}\big)^3\big),\quad 0<h-2\sqrt{\beta_{\scriptscriptstyle C}} \ll1,
	\end{eqnarray}
	and
	\begin{eqnarray}\label{eq:m1h}
		&&M_{1}(h)=C_{01} + C_{11}\big(h-2\sqrt{\beta_{\scriptscriptstyle C}}\big) + C_{21}\big(h-2\sqrt{\beta_{\scriptscriptstyle C}}\big)^2 + \mathcal{O}\big(\big(h-2\sqrt{\beta_{\scriptscriptstyle C}}\big)^3\big),\nonumber\\
		&& \quad\quad\quad \quad0<2\sqrt{\beta_{\scriptscriptstyle C}}-h \ll1,
	\end{eqnarray}
	where
	\begin{equation*}
		\begin{aligned}
			& C_{00}  = \frac{1}{2\omega_{\scriptscriptstyle R}}\bigg(4\sqrt{\beta_{\scriptscriptstyle C}} ((2 p_{00} + 2 p_{10} - b_{\scriptscriptstyle R} (2 u_{00} + u_{10} - v_{01} + (u_{10} + v_{01}) \beta_{\scriptscriptstyle C})) \omega_{\scriptscriptstyle C} + b_{\scriptscriptstyle R} (p_{10}  \\
			& \quad\quad\,\,\,\, + q_{01}) \tau_{\scriptscriptstyle R}) + 2 b_{\scriptscriptstyle R} (u_{10} + v_{01}) (1 + \beta_{\scriptscriptstyle C})^2 \omega_{\scriptscriptstyle R} \arccos\bigg(\frac{\beta_{\scriptscriptstyle C}-1}{\beta_{\scriptscriptstyle C}+1}\bigg) + 
			b_{\scriptscriptstyle R} (p_{10} + q_{01}) (4 \beta_{\scriptscriptstyle C} \\
			& \quad\quad \,\,\,\, - \tau_{\scriptscriptstyle R}^2) \log\bigg(-\frac{2\sqrt{\beta_{\scriptscriptstyle C}}+\tau_{\scriptscriptstyle R}}{2\sqrt{\beta_{\scriptscriptstyle C}}-\tau_{\scriptscriptstyle R}}\bigg)\bigg), \\
			& D_{1} = \frac{4 b_{\scriptscriptstyle R} \sqrt[4]{\beta_{\scriptscriptstyle C}}}{b_{\scriptscriptstyle L}}\bigg(r_{10} - r_{00} + b_{\scriptscriptstyle L} (u_{00} - u_{10})\bigg), \\
			& C_{10}  = 2\bigg(p_{00} + p_{10} - b_{\scriptscriptstyle R} (u_{00} + u_{10})+\frac{b_{\scriptscriptstyle R}\sqrt{\beta_{\scriptscriptstyle C}}}{\omega_{\scriptscriptstyle R}}\bigg(2 (u_{10} + v_{01}) \omega_{\scriptscriptstyle R} \arccos\bigg(\frac{\beta_{\scriptscriptstyle C}-1}{\beta_{\scriptscriptstyle C}+1}\bigg) \\
			&\quad\quad\,\,\,\,  + (p_{10} + q_{01}) \log\bigg(-\frac{2\sqrt{\beta_{\scriptscriptstyle C}}+\tau_{\scriptscriptstyle R}}{2\sqrt{\beta_{\scriptscriptstyle C}}-\tau_{\scriptscriptstyle R}}\bigg)\bigg)\bigg), \\
		\end{aligned}
	\end{equation*}
	\begin{equation*}
		\begin{aligned}	
			& D_{3}  = \frac{1}{6 b_{\scriptscriptstyle L} \sqrt[4]{\beta_{\scriptscriptstyle C}} (1 + \beta_{\scriptscriptstyle C}) \omega_{\scriptscriptstyle L} \tau_{\scriptscriptstyle L}}\bigg( -3 b_{\scriptscriptstyle R} (r_{00} - r_{10}) (1 + \beta_{\scriptscriptstyle C}) \omega_{\scriptscriptstyle L} \tau_{\scriptscriptstyle L} + b_{\scriptscriptstyle L} b_{\scriptscriptstyle R} (32 (r_{10} + s_{01}) \beta_{\scriptscriptstyle C} \quad\quad\quad\quad\quad  \\
			& \quad\quad\,\,\,\, (1 + \beta_{\scriptscriptstyle C}) + (-3 u_{10} - 35 u_{10} \beta_{\scriptscriptstyle C} - 32 v_{01} \beta_{\scriptscriptstyle C} + 3 u_{00} (1 + \beta_{\scriptscriptstyle C})) \omega_{\scriptscriptstyle L} \tau_{\scriptscriptstyle L}) \bigg), \\
			& C_{20}  = b_{\scriptscriptstyle R} (u_{10} + v_{01}) \arccos\bigg(\frac{\beta_{\scriptscriptstyle C}-1}{\beta_{\scriptscriptstyle C}+1}\bigg) + \frac{b_{\scriptscriptstyle R}}{2\omega_{\scriptscriptstyle R}}\bigg(\frac{4\sqrt{\beta_{\scriptscriptstyle C}}}{(1 + \beta_{\scriptscriptstyle C})^2 (4 \beta_{\scriptscriptstyle C} - \tau_{\scriptscriptstyle R}^2)}\bigg(4 (u_{10} + v_{01})  \\
			& \quad\quad\,\,\,\, (\beta_{\scriptscriptstyle C}-1 ) \beta_{\scriptscriptstyle C} \omega_{\scriptscriptstyle R} - (p_{10} + q_{01}) (1 + \beta_{\scriptscriptstyle C})^2 \tau_{\scriptscriptstyle R} - (u_{10} + v_{01}) (\beta_{\scriptscriptstyle C}-1) \omega_{\scriptscriptstyle R} \tau_{\scriptscriptstyle R}^2\bigg) \\
			& \quad\quad\,\,\,\, + (p_{10} + q_{01})\log\bigg(-\frac{2\sqrt{\beta_{\scriptscriptstyle C}}+\tau_{\scriptscriptstyle R}}{2\sqrt{\beta_{\scriptscriptstyle C}}-\tau_{\scriptscriptstyle R}}\bigg) \bigg), 
		\end{aligned}
	\end{equation*}	
	\begin{equation*}
		\begin{aligned}	
			& C_{01}  = \frac{1}{2}\bigg(8 (p_{00} + p_{10}) \sqrt{\beta_{\scriptscriptstyle C}} - 4 b_{\scriptscriptstyle R} \sqrt{\beta_{\scriptscriptstyle C}} (2 u_{00} + u_{10} - v_{01} + (u_{10} + v_{01}) \beta_{\scriptscriptstyle C}) + \frac{4b_{\scriptscriptstyle R}\sqrt{\beta_{\scriptscriptstyle C}}\tau_{\scriptscriptstyle R}}{\omega_{\scriptscriptstyle R}}\quad\quad\\
			&\quad\quad\,\,\,\, \bigg(p_{10} + q_{01}\bigg)+ b_{\scriptscriptstyle R} (u_{10} + v_{01}) (1 + \beta_{\scriptscriptstyle C})^2 \bigg( 2\pi - \arccos\bigg(\frac{1 - 6 \beta_{\scriptscriptstyle C} + \beta_{\scriptscriptstyle C}^2}{(1 + \beta_{\scriptscriptstyle C})^2}\bigg)\bigg) + \frac{b_{\scriptscriptstyle R}}{\omega_{\scriptscriptstyle R}} \\
			& \quad\quad\,\,\,\, (p_{10} + q_{01}) (2 \sqrt{\beta_{\scriptscriptstyle C}} - \tau_{\scriptscriptstyle R})(2 \sqrt{\beta_{\scriptscriptstyle C}} + \tau_{\scriptscriptstyle R} )\log\bigg(-\frac{2\sqrt{\beta_{\scriptscriptstyle C}}+\tau_{\scriptscriptstyle R}}{2\sqrt{\beta_{\scriptscriptstyle C}}-\tau_{\scriptscriptstyle R}}\bigg) \bigg),\\
			& C_{11}  = 2 (p_{00} + p_{10}) - b_{\scriptscriptstyle R} (2 u_{00} + u_{10} - v_{01} + (u_{10} + v_{01}) \beta_{\scriptscriptstyle C}) + \frac{b_{\scriptscriptstyle R}\tau_{\scriptscriptstyle R}}{\omega_{\scriptscriptstyle R}}\bigg(p_{10} + q_{01}\bigg)\\
			& \quad\quad\,\,\,\, +b_{\scriptscriptstyle R} (u_{10} + v_{01})\bigg( -\sqrt{(\beta_{\scriptscriptstyle L}-1)^2} + 2\sqrt{\beta_{\scriptscriptstyle C}}\bigg(2\pi - \arccos\bigg(\frac{1 - 6 \beta_{\scriptscriptstyle C} + \beta_{\scriptscriptstyle C}^2}{(1 + \beta_{\scriptscriptstyle C})^2}\bigg)\bigg) \bigg)  \\
			& \quad\quad\,\,\,\, -\frac{b_{\scriptscriptstyle R}}{\omega_{\scriptscriptstyle R}}\bigg( p_{10} + q_{01} \bigg)\bigg( \tau_{\scriptscriptstyle R} - 2\sqrt{\beta_{\scriptscriptstyle C}}\log\bigg(-\frac{2\sqrt{\beta_{\scriptscriptstyle C}}+\tau_{\scriptscriptstyle R}}{2\sqrt{\beta_{\scriptscriptstyle C}}-\tau_{\scriptscriptstyle R}}\bigg)\bigg),
		\end{aligned}
	\end{equation*}
	\begin{equation*}
		\begin{aligned}
			& 	C_{21}  = b_{\scriptscriptstyle R}\pi (u_{10} + v_{01}) + 2 b_{\scriptscriptstyle R}\sqrt{\beta_{\scriptscriptstyle C}} \bigg( \frac{\tau_{\scriptscriptstyle R}}{\omega_{\scriptscriptstyle R}(\tau_{\scriptscriptstyle R}^2-4\beta_{\scriptscriptstyle C})}(p_{10} + q_{01}) - \frac{\sqrt{(\beta_{\scriptscriptstyle C}-1)^2}}{(\beta_{\scriptscriptstyle C}^2+1)}(u_{10} + v_{01}) \bigg) \quad\quad\quad\quad\quad\quad\quad \\
			& \quad\quad\,\,\,\,- \frac{b_{\scriptscriptstyle R}}{2}(u_{10} + v_{01})\arccos\bigg(\frac{1 - 6 \beta_{\scriptscriptstyle C} + \beta_{\scriptscriptstyle C}^2}{(1 + \beta_{\scriptscriptstyle C})^2}\bigg) + \frac{b_{\scriptscriptstyle R}}{2\omega_{\scriptscriptstyle R}}(p_{10} + q_{01})\log\bigg(-\frac{2\sqrt{\beta_{\scriptscriptstyle C}}+\tau_{\scriptscriptstyle R}}{2\sqrt{\beta_{\scriptscriptstyle C}}-\tau_{\scriptscriptstyle R}}\bigg).
		\end{aligned}
	\end{equation*}			
	Further, we have that 	
	\begin{equation}\label{eq:cs}
		\begin{aligned}
			C_{00} - C_{01} & = -\frac{b_{\scriptscriptstyle R}}{2}\Big(u_{10} + v_{01}\Big)(1 + \beta_{\scriptscriptstyle C})^2 \,\Phi(\beta_{\scriptscriptstyle C}), \\
			C_{10} - C_{11} & = \left\{\begin{array}{ll}\vspace{0.2cm}
				-b_{\scriptscriptstyle R}(u_{10} + v_{01})\sqrt{\beta_{\scriptscriptstyle C}} \,\Phi(\beta_{\scriptscriptstyle C}), \quad\text{if}\quad 0<\beta_{\scriptscriptstyle C}\leq 1, \\ \vspace{0.2cm}
				2b_{\scriptscriptstyle R}(u_{10} + v_{01})(\beta_{\scriptscriptstyle C}-1-\sqrt{\beta_{\scriptscriptstyle C}}\Phi(\beta_{\scriptscriptstyle C})), \quad\text{if}\quad \beta_{\scriptscriptstyle C}> 1,
			\end{array}
			\right. \\
			C_{20} - C_{21} & = \left\{\begin{array}{ll}\vspace{0.2cm}
				-\dfrac{b_{\scriptscriptstyle R}}{2}\Big(u_{10} + v_{01}\Big)\,\Phi(\beta_{\scriptscriptstyle C}), \quad\text{if}\quad 0<\beta_{\scriptscriptstyle C}\leq 1, \\ \vspace{0.2cm}
				\dfrac{b_{\scriptscriptstyle R}}{2(\beta_{\scriptscriptstyle C}+1)^2}\Big(u_{10} + v_{01}\Big)\Big(8(\beta_{\scriptscriptstyle C}-1)\sqrt{\beta_{\scriptscriptstyle C}} - (\beta_{\scriptscriptstyle C}+1)^2\Phi(\beta_{\scriptscriptstyle C})\Big), \\
				\text{if}\quad \beta_{\scriptscriptstyle C}> 1,
			\end{array}
			\right.
		\end{aligned}
	\end{equation}	
	with
	\begin{equation*}
		\Phi(\beta_{\scriptscriptstyle C}) = 2\pi - 2\arccos\bigg(\frac{\beta_{\scriptscriptstyle C}-1}{\beta_{\scriptscriptstyle C}+1}\bigg) - \arccos\bigg(\frac{1 - 6 \beta_{\scriptscriptstyle C} + \beta_{\scriptscriptstyle C}^2}{(1 + \beta_{\scriptscriptstyle C})^2}\bigg). 
	\end{equation*}
	In this way, we will divide the proof of Theorem \ref{the:01} in two cases. The first one is when $0<\beta_{\scriptscriptstyle C}\leq1$ (i.e. the central subsystem has a real or boundary center) and the second one is when $\beta_{\scriptscriptstyle C}>1$ (i.e. the central subsystem has a virtual center). 
	
	\medskip
	
	{\it Case $0<\beta_{\scriptscriptstyle C}\leq1$}: In this case we have that  $\Phi(1)=0$ and, after some compute,  is possible to show that $\Phi'(\beta_{\scriptscriptstyle C})=0$. Therefore, $\Phi(\beta_{\scriptscriptstyle C})\equiv0$ for  $\beta_{\scriptscriptstyle C}\in(0,1]$ and,  by equations \eqref{eq:cs},  $C_{00}=C_{01}$, $C_{10}=C_{11}$ and $C_{20}=C_{21}$, i.e we can rewrite the equation \eqref{eq:m1h} as
	$$M_{1}(h)=C_{00} + C_{10}\big(h-2\sqrt{\beta_{\scriptscriptstyle C}}\big) + C_{20}\big(h-2\sqrt{\beta_{\scriptscriptstyle C}}\big)^2 + \mathcal{O}\big(\big(h-2\sqrt{\beta_{\scriptscriptstyle C}}\big)^3\big)$$
	Moreover, when 
	\begin{equation*}
		\begin{aligned}
	& p_{00} = q_{01} + b_{\scriptscriptstyle R} (u_{00} - v_{01}),\quad r_{10} = -s_{01},\quad u_{10} = -v_{01},\\
	& \quad\quad r_{00} = -s_{01} + b_{\scriptscriptstyle L} (u_{00} + v_{01})\quad\text{and}\quad p_{10} = -q_{01},
		\end{aligned}
	\end{equation*}
	we have that 
	$$C_{00}=D_{1}=C_{10}=D_{3}=C_{20}=0\quad \text{with}\quad\text{rank}\frac{\partial(C_{00},D_{1},C_{10},D_{3},C_{20})}{\partial(p_{00},r_{10},u_{10},r_{00},p_{10})}=5.$$
	This implies, by Proposition \ref{the:coef}, that $C_{00}$, $D_1$, $C_{10}$, $D_{3}$ and $C_{20}$ can be taken as free coefficients, satisfying
	$$C_{00} D_1<0,\quad D_1 C_{10}<0,\quad C_{10} D_3<0,\quad D_3 C_{20}<0$$ 
	and
	$$ |C_{00}|\ll|D_1|\ll|C_{10}|\ll|D_3|\ll|C_{20}|,$$
	such that the functions \eqref{eq:m0h} and \eqref{eq:m1h} have at least four and two positive zeros near $h=2\sqrt{\beta_{\scriptscriptstyle C}}$, respectively.
	
	Now, when
	\begin{equation*}
		\begin{aligned}
			& p_{00} = \frac{1}{2 \omega_{\scriptscriptstyle R}\lambda_1}\bigg( 4 \sqrt{\beta_{\scriptscriptstyle C}} ((p_{10} + b_{\scriptscriptstyle R} (v_{01}-u_{00} )) (\beta_{\scriptscriptstyle C}-1) \omega_{\scriptscriptstyle R} - b_{\scriptscriptstyle R} (p_{10} + q_{01}) \tau_{\scriptscriptstyle R}) - 2 ((p_{10} + b_{\scriptscriptstyle R}  \\
			& \quad \quad \,\,\,\, (v_{01}-u_{00} ))(1 + (\beta_{\scriptscriptstyle C}-6) \beta_{\scriptscriptstyle C}) \omega_{\scriptscriptstyle R} - 4 b_{\scriptscriptstyle R} (p_{10} + 
			q_{01}) \beta_{\scriptscriptstyle C} \tau_{\scriptscriptstyle R}) \arccos\bigg(\frac{\beta_{\scriptscriptstyle C}-1}{\beta_{\scriptscriptstyle C}+1}\bigg) - b_{\scriptscriptstyle R}  \\
			& \quad \quad \,\,\,\, (p_{10} + q_{01}) \bigg(-4 \beta_{\scriptscriptstyle C}^2 - \tau_{\scriptscriptstyle R}^2 +2 \sqrt{\beta_{\scriptscriptstyle C}} ((\beta_{\scriptscriptstyle C}-1)^2 + \tau_{\scriptscriptstyle R}^2)
			\arccos\bigg(\frac{\beta_{\scriptscriptstyle C}-1}{\beta_{\scriptscriptstyle C}+1}\bigg)\bigg) ,\\
			& \quad \quad \,\,\,\,\log\bigg(\frac{2\sqrt{\beta_{\scriptscriptstyle C}}+\tau_{\scriptscriptstyle R}}{\tau_{\scriptscriptstyle R}-2\sqrt{\beta_{\scriptscriptstyle C}}}\bigg) \bigg),\\
			& r_{10} = -\frac{1}{2 \omega_{\scriptscriptstyle R}\lambda_1(\beta_{\scriptscriptstyle C}+1)}\bigg( 4 \sqrt{\beta_{\scriptscriptstyle C}} (s_{01} (\beta_{\scriptscriptstyle C}^2-1) \omega_{\scriptscriptstyle R} - (p_{10} + 
			q_{01}) \omega_{\scriptscriptstyle L} \tau_{\scriptscriptstyle L} \tau_{\scriptscriptstyle R}) - 2 s_{01} (1 + \beta_{\scriptscriptstyle C})(1 + (\beta_{\scriptscriptstyle C}\\
			&  \quad \quad \,\,\,\,  -6) \beta_{\scriptscriptstyle C})\omega_{\scriptscriptstyle R}
			\arccos\bigg(\frac{\beta_{\scriptscriptstyle C}-1}{\beta_{\scriptscriptstyle C}+1}\bigg) + (p_{10} + 
			q_{01}) \omega_{\scriptscriptstyle L} \tau_{\scriptscriptstyle L} (4 \beta_{\scriptscriptstyle C} + \tau_{\scriptscriptstyle R}^2) \log\bigg(\frac{2\sqrt{\beta_{\scriptscriptstyle C}}+\tau_{\scriptscriptstyle R}}{\tau_{\scriptscriptstyle R}-2\sqrt{\beta_{\scriptscriptstyle C}}}\bigg) \bigg),
		\end{aligned}
	\end{equation*}	
	\begin{equation*}
		\begin{aligned}
			& u_{10} = -\frac{1}{2 \omega_{\scriptscriptstyle R}\lambda_1} \bigg( 4 \sqrt{\beta_{\scriptscriptstyle C}} (v_{01} (\beta_{\scriptscriptstyle C}-1) \omega_{\scriptscriptstyle R} - (p_{10} + 
			q_{01}) \tau_{\scriptscriptstyle R}) - 
			2 v_{01} (1 + (\beta_{\scriptscriptstyle C}-6) \beta_{\scriptscriptstyle C}) \omega_{\scriptscriptstyle R} \\
			& \quad \quad \,\,\,\, \arccos\bigg(\frac{\beta_{\scriptscriptstyle C}-1}{\beta_{\scriptscriptstyle C}+1}\bigg) + (p_{10} + 
			q_{01}) (4 \beta_{\scriptscriptstyle C} + \tau_{\scriptscriptstyle R}^2) \log\bigg(\frac{2\sqrt{\beta_{\scriptscriptstyle C}}+\tau_{\scriptscriptstyle R}}{\tau_{\scriptscriptstyle R}-2\sqrt{\beta_{\scriptscriptstyle C}}}\bigg) \bigg),\\
			& r_{00} = \frac{1}{2 \omega_{\scriptscriptstyle R}\lambda_1(\beta_{\scriptscriptstyle C}+1)}\bigg( 4 \sqrt{\beta_{\scriptscriptstyle C}} (-(s_{01} - b_{\scriptscriptstyle L} (u_{00} + v_{01})) ( \beta_{\scriptscriptstyle C}^2-1) \omega_{\scriptscriptstyle R} - (p_{10} + q_{01}) (b_{\scriptscriptstyle L} + b_{\scriptscriptstyle L} \beta_{\scriptscriptstyle C}  \\
			& \quad \quad \,\,\,\, - \omega_{\scriptscriptstyle L} \tau_{\scriptscriptstyle L}) \tau_{\scriptscriptstyle R})+ 2 (s_{01} - b_{\scriptscriptstyle L} (u_{00} + v_{01})) (1 + \beta_{\scriptscriptstyle C}) (1 + ( \beta_{\scriptscriptstyle C}-6) \beta_{\scriptscriptstyle C}) \omega_{\scriptscriptstyle R} \arccos\bigg(\frac{\beta_{\scriptscriptstyle C}-1}{\beta_{\scriptscriptstyle C}+1}\bigg)  \\
			& \quad \quad \,\,\,\, + (p_{10} + q_{01})(b_{\scriptscriptstyle L} + 
			b_{\scriptscriptstyle L} \beta_{\scriptscriptstyle C} - \omega_{\scriptscriptstyle L} \tau_{\scriptscriptstyle L}) (4 \beta_{\scriptscriptstyle C} + \tau_{\scriptscriptstyle R}^2) \log\bigg(\frac{2\sqrt{\beta_{\scriptscriptstyle C}}+\tau_{\scriptscriptstyle R}}{\tau_{\scriptscriptstyle R}-2\sqrt{\beta_{\scriptscriptstyle C}}}\bigg) \bigg),
		\end{aligned}
	\end{equation*}	
	with $$\lambda_1= 2  (\beta_{\scriptscriptstyle C}-1) \sqrt{\beta_{\scriptscriptstyle C}} - (1 + (\beta_{\scriptscriptstyle C}-6) \beta_{\scriptscriptstyle C}) \arccos\bigg(\frac{\beta_{\scriptscriptstyle C}-1}{\beta_{\scriptscriptstyle C}+1}\bigg),$$
	we have that $C_{00}=D_{1}=C_{10}=D_{3}=0$, 
	$$C_{20}=-\frac{b_{\scriptscriptstyle C}(p_{10} + q_{01}) \lambda_2}{2 \omega_{\scriptscriptstyle R}(1 + \beta_{\scriptscriptstyle C})^2 (4 \beta_{\scriptscriptstyle C} - \tau_{\scriptscriptstyle R}^2) \lambda_1}\quad\text{and}\quad\text{rank}\frac{\partial(C_{00},D_{1},C_{10},D_{3},C_{20})}{\partial(p_{00},r_{10},u_{10},r_{00},p_{10})}=5,$$
	with
	\begin{equation*}
		\begin{aligned}
			&\lambda_2 = 8 (\beta_{\scriptscriptstyle C}-1) \beta_{\scriptscriptstyle C} \tau_{\scriptscriptstyle R} ((\beta_{\scriptscriptstyle C}-1)^2 + \tau_{\scriptscriptstyle R}^2) - 
			4 \sqrt{\beta_{\scriptscriptstyle C}} (1 + \beta_{\scriptscriptstyle C})^2 (\beta_{\scriptscriptstyle C}-1 - \tau_{\scriptscriptstyle R}) (\beta_{\scriptscriptstyle C} + \tau_{\scriptscriptstyle R}-1)\tau_{\scriptscriptstyle R}\\
			& \quad \quad \,\,\,\, \arccos\bigg(\frac{\beta_{\scriptscriptstyle C}-1}{\beta_{\scriptscriptstyle C}+1}\bigg) + (4 \beta_{\scriptscriptstyle C} - \tau_{\scriptscriptstyle R}^2) \bigg(-2 (\
			\beta_{\scriptscriptstyle C}-1 ) \sqrt{\beta_{\scriptscriptstyle C}} ((\beta_{\scriptscriptstyle C}-1 )^2 - \tau_{\scriptscriptstyle R}^2) 
			+ (1 + \beta_{\scriptscriptstyle C})^2 \\
			& \quad \quad \,\,\,\, ((\beta_{\scriptscriptstyle C}-1)^2 + \tau_{\scriptscriptstyle R}^2) 
			\arccos\bigg(\frac{\beta_{\scriptscriptstyle C}-1}{\beta_{\scriptscriptstyle C}+1}\bigg) \log\bigg(\frac{2\sqrt{\beta_{\scriptscriptstyle C}}+\tau_{\scriptscriptstyle R}}{\tau_{\scriptscriptstyle R}-2\sqrt{\beta_{\scriptscriptstyle C}}}\bigg)\bigg).
		\end{aligned}
	\end{equation*}
	It is possible to show that the equation $\lambda_1=0$ has a unique zero in the interval $(0,1]$. Using the Newton's method we obtain that the zero is $\beta_{\scriptscriptstyle C}\approx 0.23031022687$. Moreover, there are parameters values for system $\eqref{eq:01}$ such that $\lambda_2\ne0$. For example, if $a_{\scriptscriptstyle L}=b_{\scriptscriptstyle L}=b_{\scriptscriptstyle R}=c_{\scriptscriptstyle R}=1$, $c_{\scriptscriptstyle L}=a_{\scriptscriptstyle R}=0$, $\beta_{\scriptscriptstyle R}=-4$, $\beta_{\scriptscriptstyle L} \geq \sqrt{9-4\beta_{\scriptscriptstyle C}}$ and $\beta_{\scriptscriptstyle C} \in\{1/2,1\}$, we have that the left subsystem from $\eqref{eq:01}|_{\epsilon=0}$ has a saddle at the point $(-1-\beta_{\scriptscriptstyle L},\beta_{\scriptscriptstyle L})$, the right subsystem from $\eqref{eq:01}|_{\epsilon=0}$ has a saddle at the point $(4,0)$, $\lambda_2\approx 223.526$ when $\beta_{\scriptscriptstyle C}=1$ and $\lambda_2\approx 23.0517$ when $\beta_{\scriptscriptstyle C}=1/2$. Hence, for  $p_{10} \ne -q_{01}$, we have that $C_{20}\ne 0$. This implies, by Proposition \ref{the:coef}, that we can take $C_{00}$, $D_1$, $C_{10}$, $D_{3}$ and $C_{20}$ as free coefficients, 
	such that the function \eqref{eq:m0h} has at least four positive zeros near $h=2\sqrt{\beta_{\scriptscriptstyle C}}$ and the function \eqref{eq:m1h} has exactly two positive zeros near $h=2\sqrt{\beta_{\scriptscriptstyle C}}$. Therefore, the number of limit cycles from system \eqref{eq:01} that can bifurcate of the periodic annulus near $h=2\sqrt{\beta_{\scriptscriptstyle C}}$ when  $0<\beta_{\scriptscriptstyle C}\leq1$ is at least six.	
	
	\medskip
	
	{\it Case $\beta_{\scriptscriptstyle C}>1$}: In this case, 
	if we take $v_{01}=-u_{10}$ in \eqref{eq:cs}, we obtain $C_{00}=C_{01}$, $C_{10}=C_{11}$ and $C_{20}=C_{21}$.
	Moreover, when
	\begin{equation*}
		\begin{aligned}
			& p_{00} = q_{01} + b_{\scriptscriptstyle R} (u_{00} + u_{10}),\quad r_{10} = r_{00} + b_{\scriptscriptstyle L} (u_{10} -u_{00} ), \\
			& \quad\quad p_{10} = -q_{01}\quad\text{and}\quad u_{10}\ne-\frac{1}{b_{\scriptscriptstyle L}}(r_{00} + s_{01}) + u_{00},
		\end{aligned}
	\end{equation*}
	we have that 
	$$C_{00}=D_{1}=C_{10}=0,\quad D_{3}\ne0\quad \text{with}\quad\text{rank}\frac{\partial(C_{00},D_{1},C_{10},D_{3})}{\partial(p_{00},r_{10},p_{10},u_{10})}=4.$$
	This implies, by Proposition \ref{the:coef}, that we can take $C_{00}$, $D_1$, $C_{10}$ and $D_{3}$ as free coefficients, such that the functions \eqref{eq:m0h} and \eqref{eq:m1h} have at least three and one positive zeros near $h=2\sqrt{\beta_{\scriptscriptstyle C}}$, respectively. Therefore, the number of limit cycles from system \eqref{eq:01} that can bifurcate of the periodic annulus near $h=2\sqrt{\beta_{\scriptscriptstyle C}}$ is at least four.
\end{proof}


\section{Acknowledgments}

The first author is partially supported by S\~ao Paulo Research Foundation (FAPESP) grants 19/10269-3 and 18/19726-5.
The second author is supported by CAPES grant 88882.434343/2019-01.


\end{document}